\documentclass[a4paper,twoside,11pt]{amsart}
\usepackage[utf8]{inputenc}
\usepackage[T1]{fontenc}
\usepackage[british]{babel}
\usepackage{amssymb,amsmath,amsthm,bbm,comment}
\usepackage{enumerate}
\usepackage{verbatim}
\usepackage{lmodern}
\usepackage[dvipsnames]{xcolor}
\usepackage{microtype}
\usepackage{mathtools}
\usepackage[shortcuts]{extdash} 			
\usepackage{mparhack}	
\usepackage{hyperref}
\usepackage{tikz-cd}

\newcommand{\C}{\ensuremath{\mathbb{C}}}%
\renewcommand{\H}{\mathcal H}
\newcommand{\B}{\mathfrak{B}}

		\newcommand{\NN}{\mathbb{N}}	
			
		\newcommand{\RR}{\mathbb{R}}

		\newcommand{\CH}{\mathcal{H}}	
			
\newcommand{\CK}{K}         			
			
\newcommand{\CO}{\mathcal{O}}

\newcommand{\CU}{\mathcal{U}}

\DeclarePairedDelimiter\abs{\lvert}{\rvert}		
\DeclarePairedDelimiter\norm{\lVert}{\rVert}		
\DeclarePairedDelimiter\paren{(}{)}			
\DeclarePairedDelimiter\braces{\{}{\}}			

\newcommand{\bigpar}[1]{\paren[\big]{#1}}	\newcommand{\bigparen}[1]{\paren[\big]{#1}}
	\newcommand{\Bigparen}[1]{\paren[\Big]{#1}}

\newcommand{\tensor}{\otimes}
\newcommand{\supp}{{\rm Supp}}
\newcommand{\diam}{{\rm diam}}
\newcommand{\id}{{\rm id}}
\newcommand{\Ad}{{\rm Ad}}

\newcommand{\roestar}[1]{\C_{\rm Roe}[#1]}
\newcommand{\roecstar}[1]{C^*_{\rm Roe}(#1)}
\newcommand{\roecstarmax}[1]{C^*_{\rm Roe,max}(#1)}
\newcommand{\fpstar}[1]{\C_{\rm fp}[#1]}
\newcommand{\fpcstar}[1]{C^*_{\rm fp}(#1)}

\newcommand{\halfline}{\mathbb{R}_{\geq 1}}
\newcommand{\spanc}{{\rm Span}_\C}
\newcommand{\rtimesalg}{\rtimes_{\rm alg}}
\newcommand{\rtimesmax}{\rtimes_{\rm max}}

\definecolor{darkgreen}{rgb}{0.0, 0.5, 0.0}

\renewcommand\epsilon\varepsilon

\theoremstyle{plain}
\newtheorem{thmA}{Theorem}

\newtheorem{corA}[thmA]{Corollary}

\newtheorem{thm}{Theorem}[section]
\newtheorem{lemma}[thm]{Lemma}
\newtheorem{lem}[thm]{Lemma}

\newtheorem{prop}[thm]{Proposition}
\newtheorem{cor}[thm]{Corollary}
\newtheorem*{cor*}{Corollary}

\theoremstyle{definition}
\newtheorem{definition}[thm]{Definition}
\newtheorem{de}[thm]{Definition}

\newtheorem{rmk}[thm]{Remark}

\newtheorem{assumption}[thm]{Assumption}
\newtheorem*{notation}{Notation}

\numberwithin{equation}{section}

\author{Tim de Laat}
 \address{Tim de Laat
 \newline Department of Mathematics, Kiel University
 \newline Heinrich-Hecht-Platz 6, 24118 Kiel, Germany}
 \email{delaat@math.uni-kiel.de}

\author{Federico Vigolo}
\address{Federico Vigolo
\newline Mathematisches Institut, Georg-August Universit\"at G\"ottingen
\newline Bunsenstra\ss{}e 3, 37073 G\"ottingen, Germany}
\email{federico.vigolo@uni-goettingen.de}

\author{Jeroen Winkel}
 \address{JEROEN WINKEL,
 \newline Mathematical Institute, University of Münster,
 \newline 48149 Münster, Germany}
 \email{winkeljeroen@gmail.com}

\title{Dynamical propagation and Roe algebras of warped spaces}

\begin{document}

\begin{abstract}
Given a non-singular action $\Gamma \curvearrowright (X,\mu)$, we define the $*$-algebra $\mathbb C_{\rm fp}[\Gamma \curvearrowright X]$ of operators of finite dynamical propagation associated with this action. This assignment is completely canonical and depends only on the class of measures of $\mu$.
We prove that the algebraic crossed product $L^{\infty}X \rtimesalg \Gamma$ surjects onto $\mathbb C_{\rm fp}[\Gamma \curvearrowright X]$ and that this surjection is a $\ast$-isomorphism whenever the action is essentially free.
As a consequence, we canonically characterize ergodicity and strong ergodicity of the action in terms of structural properties of $\mathbb C_{\rm fp}[\Gamma \curvearrowright X]$ and its closure.
We also use these techniques to describe the Roe algebra of a warped space in terms of the Roe algebra of the (non-warped) space and the group action. We apply this result to Roe algebras of warped cones.
\end{abstract}

\maketitle

\section{INTRODUCTION} \label{sec:introduction}

The Roe algebra and the uniform Roe algebra of a metric space are $C^*$-algebras that encode coarse geometric features of the space. They essentially go back to \cite{RoeIndexI,RoeIndexII,RoeMAMS} and have been studied extensively in the fields of index theory, operator algebras, and coarse geometry. In a recent breakthrough, it was shown that in the (fundamentally important) setting of uniformly locally finite metric spaces, the uniform Roe algebra completely remembers the coarse equivalence class of the underlying space \cite[Theorem 1.2]{baudier2022uniform}. Partial results in this direction are also known for the (non-uniform) Roe algebra \cite{SW13a,li2020measured}. Very recent work of Mart\'\i nez and the second-named author \cite{martinez2023rigidity,martinez2023rigiditybounded} has completed the picture by proving a strong form of unconditional rigidity in the setting of locally compact (extended) metric spaces.

Elements of the Roe algebra can be approximated by locally compact operators with finite propagation. When considering the coarse geometry of group actions, it is natural to replace the notion of finite propagation by its dynamical analogue, i.e.~finite dynamical propagation, which goes back to \cite{li2020markovian}. 
In analogy with the wealth of results that hold for Roe algebras of metric spaces, it is reasonable to expect that the algebra of operators of finite dynamical propagation should contain a great deal of information about the large-scale behaviour of the dynamics of the group action.
The first aim of this article is to initiate the study of such algebras and to substantiate this expectation by showing that they provide a canonical operator-algebraic characterisation of strong ergodicity.

In what follows, let $\alpha\colon \Gamma\curvearrowright (X,\mu)$ be a non-singular action of a countable discrete group on a standard measure space; see Section \ref{ssec:actions} for precise definitions and conventions.

\begin{definition}
Let $T\in B(L^2X)$.
We say that $T$ has \emph{finite dynamical propagation} (or \emph{finite $\alpha$-propagation}) if there is a finite subset $S\subseteq\Gamma$ such that for every measurable set $A\subseteq X$ and $\eta\in L^2A$, the function $T\eta$ is supported on $S\cdot A\coloneqq \bigcup_{s\in S}s\cdot A$. We also say that \emph{$T$ is supported on $S$}.
\end{definition}

The set $\fpstar{\Gamma\curvearrowright X}$ of operators of finite dynamical propagation is a $\ast$\=/subalgebra of $B(L^2X)$. Its norm closure in $B(L^2X)$ is a $C^*$\=/algebra, which we denote by $\fpcstar{\Gamma\curvearrowright X}$. 
This $C^*$-algebra can be thought of as an analogue of a Roe algebra, where the geometric structure of $\Gamma\curvearrowright X$ is inherited from that of $\Gamma$ via the orbit structure of the dynamical system; see Remark~\ref{rmk:alpha_propagation}.

\begin{rmk}\label{rmk:algebras invariant of measureclass}
We observe that ---up to canonical isomorphism--- the algebra $\fpstar{\Gamma\curvearrowright X}$ only depends on the measure class of $\mu$. More precisely, if $\nu\sim\mu$ is an equivalent measure, then pointwise multiplication by $(\frac{d\mu}{d\nu})^{\frac 12}$ defines a unitary operator $U_{\mu,\nu}\colon L^2(X,\mu)\to L^2(X,\nu)$ and hence a natural $\ast$\=/isomorphism $\Ad(U_{\mu,\nu})\colon  B(L^2(X,\mu))\to B(L^2(X,\nu))$ mapping $T$ to $U_{\mu,\nu}TU_{\mu,\nu}^*$. 
This restricts to a $\ast$\=/isomorphism between $\fpstar{\Gamma\curvearrowright(X,\mu)}$ and $\fpstar{\Gamma\curvearrowright (X,\nu)}$ because pointwise multiplication by an a.e. positive function does not affect dynamical propagation. 
In turn, $\Ad(U_{\mu,\nu})$ must also restrict to an \hyphenation{i-so-mor-phism}isomorphism $\fpcstar{\Gamma\curvearrowright(X,\mu)}\cong \fpcstar{\Gamma\curvearrowright (X,\nu)}$ of the closures. This also justifies the choice of dropping the dependence on the measure from the notation.
\end{rmk}

Let $\pi_\alpha\colon\Gamma\to U(L^2X)$ be the Koopman representation associated with the action $\alpha\colon\Gamma\curvearrowright (X,\mu)$. For all $\gamma\in\Gamma$, the operator $\pi_\alpha(\gamma)$ has finite dynamical propagation. Together with the $\ast$\=/representation $L^\infty X\to B(L^2 X)$ by multiplication operators, this induces a natural $\ast$\=/homomorphism $\varPhi\colon L^\infty X\rtimesalg\Gamma\to \fpstar{\Gamma\curvearrowright X}$, where $L^\infty X\rtimesalg\Gamma$ denotes the algebraic crossed product (see Section~\ref{ssec:crossed products}).
Our first result concerns this comparison map and shows that in the most interesting cases, it is an isomorphism.

\begin{thmA}\label{thmA:finite propagation as crossed product}
The $\ast$\=/homomorphism $\varPhi\colon L^\infty X\rtimesalg\Gamma\to\fpstar{\Gamma\curvearrowright X}$ is \hyphenation{sur-jec-tive}surjective. If the action $\alpha \colon \Gamma \curvearrowright X$ is essentially free, then $\varPhi$ is a $\ast$\=/isomorphism.
\end{thmA}

In the special case of $\Gamma$ acting on itself by left translation, Theorem \ref{thmA:finite propagation as crossed product} shows, after taking $C^*$-completions, that the algebra $\fpcstar{\Gamma\curvearrowright \Gamma}$ concides with the classical uniform Roe algebra $C^*_u(\vert \Gamma \vert) \cong \ell^{\infty}(\Gamma) \rtimes_r \Gamma$.

Using Theorem \ref{thmA:finite propagation as crossed product}, we can study the structure of $\fpstar{\Gamma\curvearrowright X}$ and its closure $\fpcstar{\Gamma\curvearrowright X}$. In particular, it is straightforward to observe that the commutant $\fpstar{\Gamma\curvearrowright X}'$ coincides with the space $(L^\infty X)^\Gamma$ of $\Gamma$\=/invariant bounded functions (Proposition~\ref{prop:commutant of fpstar}).

We now deduce the first relations between the dynamics of the action and operator-algebraic properties of $\fpstar{\Gamma\curvearrowright X}$. Recall that a $\ast$\=/subalgebra $A\leq B(L^2X)$ is irreducible if there are no non-trivial $A$-invariant closed subspaces of $L^2X$. Equivalently, $A$ is irreducible if and only if it has trivial commutant or, by von Neumann's bicommutant theorem, $A$ is dense in $B(L^2X)$ in the weak operator topology. Furthermore, recall that $\alpha \colon \Gamma \curvearrowright X$ is \emph{ergodic} if, up to measure zero, there are no nontrivial invariant measurable subsets of $X$ (equivalently, the only $\Gamma$-invariant functions are the constant ones). As a consequence of Theorem \ref{thmA:finite propagation as crossed product}, we recover the following classical fact, which essentially goes back to \cite[Lemma 12.2.4]{murrayvonneumann1936}.

\begin{corA} \label{corA:ergodicity}
 The following are equivalent:
 \begin{enumerate}
  \item $\Gamma\curvearrowright X$ is ergodic;
  \item $\fpstar{\Gamma\curvearrowright X}'=\C$;
  \item $\fpstar{\Gamma\curvearrowright X}$ is dense in $ B(L^2 X)$ in the weak operator topology;
  \item $\fpstar{\Gamma\curvearrowright X}< B(L^2 X)$ is irreducible.
 \end{enumerate}
The analogous equivalences for the norm closure $\fpcstar{\Gamma\curvearrowright X}$ hold as well.
\end{corA}

If an irreducible $C^*$-algebra contains a non-zero compact operator, then it must contain the whole algebra of compact operators (see e.g.~\cite[Theorem 2.4.9]{murphy2014c}). This nicely complements some known facts on strongly ergodic actions. By definition, a non-singular action $\Gamma\curvearrowright(X,\mu)$ is \emph{strongly ergodic} if, after passing to any equivalent probability measure $\nu\sim\mu$, the following holds: For every sequence of measurable subsets $A_n\subseteq X$ such that $\nu(A_n\bigtriangleup \gamma \cdot A_n)\to 0$ for every $\gamma\in\Gamma$, we have $\nu(A_n)(1-\nu(A_n))\to 0$. The notion of strong ergodicity originated in \cite{schmidt1980asymptotically,schmidt1981amenability,connes1980property}, and it has, for instance, been significant in various structural results for von Neumann algebras arising from group actions and measurable equivalence relations. It is also strongly related to spectral gap properties, and for measure-preserving actions it can, in fact, be shown to be equivalent to a ``local spectral gap'' \cite{BIG17,marrakchi2018strongly} (see also \cite{li2021asymptotic}). However, up to this point, no purely operator-algebraic characterization of strong ergodicity was known.

From the (known) relations between strong ergodicity and spectral gap properties, it follows that an ergodic action $\Gamma\curvearrowright X$ is strongly ergodic if and only if $\fpcstar{\Gamma\curvearrowright X}$ contains certain non-trivial rank-one projections (see \cite{houdayer2017strongly,marrakchi2018strongly}, or see \cite[Section 4.3]{li2020markovian} for a direct implication). Together with the irreducibility of $\fpcstar{\Gamma\curvearrowright X}$, we deduce the following.

\begin{corA} \label{corA:strong ergodicity}
 The action $\Gamma\curvearrowright (X,\mu)$ is strongly ergodic if and only if $\fpcstar{\Gamma\curvearrowright X}$ contains the compact operators $\CK(L^2X)$.
\end{corA}

\begin{rmk}
From Corollaries \ref{corA:ergodicity} and \ref{corA:strong ergodicity}, it follows that $\Gamma\curvearrowright X$ is ergodic but not strongly ergodic if and only if $\fpcstar{\Gamma\curvearrowright X}$ is irreducible and does not contain any non-zero compact operator.
\end{rmk}

\begin{rmk}
    By Remark~\ref{rmk:algebras invariant of measureclass}, we know that the purely $C^*$-algebraic characterization of strong ergodicity of Corollary~\ref{corA:strong ergodicity} is completely canonical and valid for any choice of $\mu$ in the measure class. In particular, it even holds whenever $\mu$ is an infinite measure, while verifying the usual (and to the authors' knowledge thus far only) definition of strong ergodicity would require to pass to an equivalent probability measure.
\end{rmk}

Another important reason to study the algebras of operators of finite dynamical propagation is that they are instrumental for the understanding of Roe algebras of spaces with warped metrics. This is the topic of the second part of this article, which is largely independent of the first part. Results in this direction were already proved in \cite{higginbotham2019coarse,winkel2021geometric} and, to a lesser extent, in \cite{li2020markovian}. Background on Roe algebras and warped metrics is given in Section~\ref{sec:roe alg warped spaces}.

Given a unitary representation $\pi\colon \Gamma\to B(\CH)$ normalizing a $\ast$\=/algebra $B \leq B(\mathcal{H})$ (i.e.~$\pi(\Gamma)$ normalizes $B$), the vector space spanned by the elements of the form $\pi(\gamma)b$, with $\gamma \in \Gamma$ and $b \in B$, is a $\ast$\=/subalgebra of $B(\mathcal{H})$, which we denote by $\Gamma\cdot B$.
Since $\pi$ normalizes $B$, the $\ast$\=/algebra $\Gamma \cdot B$ also coincides with the vector space spanned by the elements $b\pi(\gamma)$. 
The mapping $(b,\gamma)\mapsto b\pi(\gamma)$ defines a $\ast$\=/homomorphism from the algebraic crossed product $B\rtimesalg\Gamma$ to $B(\CH)$, where the crossed product structure is defined by the conjugation in $B(\CH)$. The image of this homomorphism is $\Gamma\cdot B$.

In particular, we may rephrase the first part of Theorem~\ref{thmA:finite propagation as crossed product} by saying that the Koopman representation $\pi_\alpha$ (see Section \ref{ssec:actions}) normalizes $L^\infty X< B(L^2 X)$ and $\Gamma\cdot (L^\infty X)$ coincides with $\fpstar{\Gamma\curvearrowright X}$.

Now, let $(Y,d,\mu)$ be a proper (unbounded) metric measure space such that $L^2Y$ is an ample geometric $Y$-module, let $\alpha\colon\Gamma\curvearrowright Y$ be a non-singular Lipschitz action, and let $\delta_\Gamma$ denote the warped metric induced by $\alpha$ on $Y$ (see Section~\ref{ssec:warped metric}). The metric $\delta_\Gamma$ encodes features of the metric space $(Y,d)$ and of the group action $\alpha$. Warped metrics are a central source of exotic examples of continuous metric spaces \cite{roe2005warped,Sawicki2018warped,Nowak2017warped,Vig19,li2020markovian} and of families of geometrically rigid expanders and superexpanders \cite{Vig19,delaat2019superexpanders,sawicki2020super,fisher2019rigidity,
vigolo2019discrete,arzhantseva2021origami}.

Let $\roestar{Y,d}<B(L^2Y)$ denote the Roe $\ast$\=/algebra, and let $\roecstar{Y,d}$ denote the Roe $C^{\ast}$-algebra (usually and in the rest of this article just called the Roe algebra), i.e.~the norm closure of $\roestar{Y,d}$. One can verify that $\roestar{Y,d}$ and $\roecstar{Y,d}$ are normalized by the Koopman representation $\pi_\alpha(\Gamma)$; see Lemma~\ref{lem:the action normalizes}. The following is our second main result, which relates the Roe algebra of $(Y,\delta_\Gamma)$ to that of $(Y,d)$. 

\begin{thmA} \label{thmA:roe alg of warped spaces}
 Let $(Y,d,\mu)$ and $\alpha \colon \Gamma \curvearrowright Y$ be as above, and suppose that $Y$ has bounded geometry. Then
 \[
 \roestar{Y, \delta_\Gamma}=\Gamma\cdot \roestar{Y, d}
 \quad\text{and}\quad
 \roecstar{Y, \delta_\Gamma}=\overline{\Gamma\cdot\roecstar{Y, d}}^{\,\norm{\cdot}},
 \]
 where $\roestar{Y, \delta_\Gamma}$ and $\roecstar{Y, \delta_\Gamma}$ denote the Roe $\ast$-algebra and Roe algebra of the warped space $(Y,\delta_\Gamma)$.
\end{thmA}

\begin{corA} \label{corA:warped roe as crossed product}
 In the setting of Theorem \ref{thmA:roe alg of warped spaces}, there is a surjective $\ast$\=/homomorphism
 \[
  \varPsi\colon\roestar{Y,d}\rtimesalg\Gamma\to \roestar{Y,\delta_\Gamma}.
 \]
 Taking the completion, we may thus naturally identify $\roecstar{Y,\delta_\Gamma}$ with a quotient of a certain $C^*$\=/crossed product.
\end{corA}

The situation is even cleaner for maximal Roe algebras:
\begin{corA} \label{corA:warped roe as crossed product-maximal}
    Under the same assumptions as in Corollary \ref{corA:warped roe as crossed product}, there is a surjective $\ast$\=/homomorphism
    \[
    \roecstarmax{Y,d}\rtimesmax\Gamma\to \roecstarmax{Y,\delta_\Gamma}.
    \]
\end{corA}

\begin{rmk} \
    \begin{enumerate}
        \item For actions on discrete metric spaces, Corollary~\ref{corA:warped roe as crossed product-maximal} is essentially proved in \cite[Theorem 5.2]{higginbotham2019coarse}.
        \item We also prove analogues of Theorem~\ref{thmA:roe alg of warped spaces} and Corollary~\ref{corA:warped roe as crossed product} for the algebras $\fpstar{Y,\delta_\Gamma}$ and $\fpcstar{Y,\delta_\Gamma}$ of operators of finite propagation.
    \end{enumerate}
\end{rmk}

Finally, we study the consequences of the above results in the setting of warped cones. Given an action $\Gamma \curvearrowright X$, where $(X,d)$ is a compact metric space, the warped cone $\CO_\Gamma X$, as introduced in \cite{roe2005warped}, is obtained by warping the open cone $Y=\CO X$ along an action of the form $\alpha\times \id \colon \Gamma\curvearrowright \CO X$; see Section \ref{sec:warped cones} for details. Warped cones encode coarse geometric features of the space $X$ and of the dynamics of the action. In this setting, we may employ the cone structure to understand the kernels of the $\ast$-homomorphisms we constructed more precisely. Specifically, we prove the following (we refer to the Notation on page \pageref{notation module cones} for details).

\begin{thmA} \label{thmA:roe warped cone}
Let $(X,d,\mu)$ be a compact metric space of diameter at most $2$ with a $\sigma$-finite measure of full support, let $\alpha\colon\Gamma\curvearrowright X$ be a free non-singular Lipschitz action, suppose that $\CO X$ has bounded geometry and that $\CH_{\CO X}$ has the operator norm localization property. Then the surjective $\ast$\=/homomorphism 
\[\varPsi\colon\roestar{\CO X}\rtimesalg \Gamma\to\roestar{\CO_\Gamma X}\]
descends to an isomorphism
 \[
  \frac{\roestar{\CO X}}{\CK\paren{L^2(\CO X)}\cap \roestar{\CO X}}
  \rtimesalg \Gamma
  \cong \frac{\roestar{\CO_\Gamma X}}{\CK\paren{L^2(\CO X)}\cap \roestar{\CO_\Gamma X}}
 \]
 and to an injective $\ast$-homomorphism with dense image:
 \[
  \frac{\roecstar{\CO X}}{\CK\paren{L^2(\CO X)}}
  \rtimesalg \Gamma
  \to \frac{\roecstar{\CO_\Gamma X}}{\CK\paren{L^2(\CO X)}}.
 \]
 Thus, the quotient $\frac{\roecstar{\CO_\Gamma X}}{\CK\paren{L^2(\CO X)}}$ may be viewed as a completion of $\frac{\roecstar{\CO X}}{\CK\paren{L^2(\CO X)}}\rtimesalg \Gamma$.
\end{thmA}

\begin{rmk}
    A related result was proved in \cite[Lemma 11.8]{winkel2021geometric}. There, instead of Roe $\ast$-algebras and compact operators, the quotient of $\fpstar{\CO X}$ by the algebra of operators with finite support is studied.
\end{rmk}

Our interest in Theorem~\ref{thmA:roe warped cone} stems from a desire to understand the analytic $K$-theory of warped cones, which is part of ongoing work. As with Theorem~\ref{thmA:roe alg of warped spaces}, also Theorem~\ref{thmA:roe warped cone} interacts nicely with maximal norms. In particular, it implies the following.

\begin{corA} \label{corA:roe alg of warped cones-maximal}
    Given $\Gamma\curvearrowright X$ as in Theorem~\ref{thmA:roe warped cone}, $\varPsi$ induces an isomorphism
    \[
      \frac{\roecstarmax{\CO X}}{\CK\paren{L^2(\CO X)}}
      \rtimesmax \Gamma
      \cong \frac{\roecstarmax{\CO_\Gamma X}}{\CK\paren{L^2(\CO X)}}.
    \]
\end{corA}
The above corollary can be viewed as an analogue for warped cones of \cite[Proposition 2.8]{OyonoOyonoYu2009}.
\begin{rmk}
    Corollary~\ref{corA:roe alg of warped cones-maximal} is closely related to \cite[Theorem 5.2]{higginbotham2019coarse}. There, an analogous isomorphism is proved for warped spaces arising from coarsely discontinuous actions on discrete metric spaces. Since $\Gamma\curvearrowright X$ is free, the action $\Gamma\curvearrowright \CO X$ is indeed coarsely discontinuous; however, \cite[Theorem 5.2]{higginbotham2019coarse} does not apply because $\CO X$ is very much not discrete: In most cases of interest (e.g.~ergodic actions on compact manifolds), there is no natural way to discretize $\CO X$ while retaining the $\Gamma$-action.

    Under the additional (strong) assumption that $\Gamma$ is amenable, one can show that the resulting warped cone has Yu's property A (see \cite[Proposition~ 3.1]{roe2005warped} or \cite[Theorem~5.9]{higginbotham2019coarse}). In this case, the reduced and the maximal Roe algebra (resp.\ crossed products) coincide. Hence, one also obtains an isomorphism (cf.~\cite[Corollary 5.10]{higginbotham2019coarse})
    \[
      \frac{\roecstar{\CO X}}{\CK\paren{L^2(\CO X)}}
      \rtimes_{\rm red} \Gamma
      \cong \frac{\roecstar{\CO_\Gamma X}}{\CK\paren{L^2(\CO X)}}.
    \]
\end{rmk}

\subsection*{Structure of the paper}
Section~\ref{sec:prelims} contains preliminaries and notation that will be used throughout; preliminaries on Roe algebras and warped spaces are given in Section~\ref{sec:prelims of roe algebras and warped metrics}. Theorem~\ref{thmA:finite propagation as crossed product} is proved in Section \ref{sec:main thm A}, Theorem \ref{thmA:roe alg of warped spaces} is proved in Section \ref{sec:roe alg warped spaces}, and Theorem~\ref{thmA:roe warped cone} is proved in Section~\ref{sec:warped cones}.

\subsection*{Acknowledgements} 
We are grateful to Thomas Weighill for pointing out \cite{higginbotham2019coarse} to us and for drawing our attention to the possible consequences of our results for maximal Roe algebras and full crossed products. We thank Ralf Meyer for his valuable input and Christos Kitsios for his comments on an early version of this paper.

This research was funded by the Deutsche Forschungsgemeinschaft -- Project-ID 427320536 -- SFB 1442; Project-ID 398436923 -- GRK 2491; as well as under Germany’s Excellence Strategy -- EXC 2044 -- 390685587, Mathematics M\"unster: Dynamics -- Geometry -- Structure.

\section{PRELIMINARIES: ACTIONS AND ALGEBRAS}\label{sec:prelims}

\subsection{Notation and conventions} \label{subsec:notationandconventions}
$\Gamma$ always denotes a countable discrete group. We generally use $X$ for measure spaces or compact metric spaces, while we use $Y$ for proper metric spaces which typically have infinite diameter.

Given a $\ast$\=/subalgebra $B\leq B(\CH)$, an operator $T\in B(\CH)$ \emph{normalizes} $B$ if $TbT^*\in B$ for every $b\in B$.
Given $A,C\subseteq  B(\CH)$, we denote by $A\circ C$ the set of operators $\{ac\mid a\in A, c\in C\}$, and by $\spanc\paren{A}$ the vector subspace generated by $A$. These are in general not subalgebras. 
However, if $A,B\leq\B(\CH)$ are $\ast$\=/subalgebras such that $A$ normalizes $B$, then the vector space $\spanc\paren{A\circ B}=\spanc\paren{B\circ A}$ is a $\ast$\=/algebra, which we may denote by $A\cdot B$ or $B\cdot A$.
In particular, this applies to $A=\pi(\Gamma)$ if $\pi\colon \Gamma\to B(\CH)$ is a unitary representation normalizing $B$ (i.e.~such that $\pi(\Gamma)$ normalizes $B$).

\begin{notation}
 If $\pi$ is clear from the context, we simply denote by $\Gamma\cdot B$ the $\ast$\=/algebra $\pi(\Gamma)\cdot B=B\cdot\pi(\Gamma)$ (see also Section \ref{sec:introduction}).
\end{notation}

Note that when $B$ is not unital, $\Gamma\cdot B$ need not be the $\ast$\=/algebra generated by $\pi(\Gamma)$ and $B$.

\subsection{Actions on measure spaces} \label{ssec:actions}
Throughout, by a \emph{measure space} $(X,\mu)$, we mean a standard Borel space $X$ (i.e.~the underlying $\sigma$-algebra is the Borel $\sigma$-algebra defined by some complete and separable metric) equipped with a $\sigma$-finite measure. When the measure is clear from the context, we write $L^pX$ for $L^p(X,\mu)$.

We consider measurable group actions $\alpha\colon\Gamma\curvearrowright (X,\mu)$ of countable discrete groups $\Gamma$. We do not require $\mu$ to be invariant under the action, but we do assume it to be quasi-invariant, i.e.~the $\Gamma$-action is measure class preserving. Such an action is said to be \emph{non-singular}. A non-singular action $\alpha$ induces a linear action $\Gamma\curvearrowright L^\infty X$ by precomposition:
\[
    \gamma \cdot f(x)\coloneqq f(\gamma^{-1}\cdot x).
\]

For every $\gamma \in \Gamma$ and $x \in X$, let
\(
 r(\gamma,x)\coloneqq \frac {d\gamma_*\mu}{d\mu}(x)
\)
be the Radon--Nikodym cocycle, which satisfies the cocycle identity 
\begin{equation}\label{eq:cocycle}
    r(\gamma_1\gamma_2,x)=r(\gamma_1,x)r(\gamma_2,\gamma_1^{-1} \cdot x).    
\end{equation}
The \emph{Koopman representation} associated with the action $\alpha\colon \Gamma\curvearrowright (X,\mu)$ is the unitary representation $\pi_\alpha\colon\Gamma\to U (L^2X)$ defined by
\begin{equation}\label{eq:koopman}
\pi_\alpha(\gamma)\xi(x)\coloneqq r(\gamma,x)^{\frac 12}\xi(\gamma^{-1}\cdot x)
\end{equation}
for $\xi\in L^2X$ and $\gamma\in \Gamma$.

An action $\alpha\colon\Gamma\curvearrowright (X,\mu)$ is \emph{essentially free} if for every $\gamma \in \Gamma \smallsetminus \{e\}$, the set of points fixed by $\gamma$ has measure $0$. In this case we can choose a co-null $\Gamma$-invariant subset $X_0$ such that the restriction of $\alpha$ to $X_0$ is free. Thus, from a measure-theoretic point of view, there is no difference between free and essentially free actions. On the other hand, it can be convenient to consider essentially free actions when one wishes to preserve some extra structure, e.g.~when $X$ is a compact topological space and the action is by continuous maps.

\subsection{Crossed products}\label{ssec:crossed products}
We now recall some background on algebraic crossed products. For details, we refer to e.g. \cite{brownozawa2008}, \cite{williams2007crossed}.

Let $\alpha \colon \Gamma \curvearrowright A$ be an action by $\ast$-homomorphisms on a $\ast$-algebra. The associated \emph{algebraic crossed product} $A\rtimesalg\Gamma$ is the $\ast$\=/algebra whose elements are finite formal sums of the form
\(\sum_{\gamma\in\Gamma}(a_\gamma,\gamma)\)
with $a_\gamma\in A$; in the literature this algebra is often denoted by $C_c(\Gamma,A)$. The addition is performed component-wise, i.e.~$(a,\gamma)+(a',\gamma)\coloneqq(a+a',\gamma)$, the multiplication is given by $(a,\gamma)(a',\gamma')\coloneqq(a\alpha_\gamma(a'),\gamma\gamma')$, and the star operation is defined by $(a,\gamma)^*\coloneqq (\alpha_{\gamma^{-1}}(a^*),\gamma^{-1})$.

Non-degenerate $*$-representations of the algebraic crossed product correspond to non-degenerate \emph{covariant representations} of the pair $(A,\Gamma)$. A covariant representation $\pi$ of $(A,\Gamma)$ is a pair $(\pi_A,\pi_{\Gamma})$ consisting of a $*$-representation $\pi_A \colon A\to  B(\H)$ and a unitary representation $\pi_\Gamma \colon \Gamma\to U(\H)$ (on the same Hilbert space $\mathcal{H}$) satisfying
\begin{equation}\label{eq:covariant rep}
 \pi_A(\alpha_\gamma (a))=\pi_\Gamma(\gamma)\pi_A(a)\pi_\Gamma(\gamma^{-1})
\end{equation}
for all $a \in A$ and $\gamma \in \Gamma$. Explicitly, the $\ast$-representation associated with $(\pi_A,\pi_\Gamma)$ is given by
\[
\pi(a,\gamma)\coloneqq \pi_A(a)\pi_\Gamma(\gamma).
\]

Observe that $\pi_\Gamma(\Gamma)$ normalizes $\pi_A(A)$ and the image $\pi(A\rtimesalg\Gamma)\leq B(\CH)$ (where $\pi$ is the $\ast$\=/representation corresponding to $(\pi_A,\pi_{\Gamma})$) coincides with $\Gamma\cdot \pi_A(A)$. Conversely, given $A\leq B(\CH)$ and a unitary representation $\pi_\Gamma\colon \Gamma\to U(\CH)$ normalizing $A$, an element $\gamma\in\Gamma$ acts on $A$ via conjugation by $\pi_\Gamma(\gamma)$. We thus obtain an algebraic crossed product $A\rtimesalg \Gamma$ and a covariant representation $(\pi_A,\pi_\Gamma)$ such that the image under the corresponding $\ast$\=/representation is the $\ast$-algebra $\Gamma\cdot A$ (here $\pi_A$ is the inclusion).

Now, let $\alpha \colon \Gamma \curvearrowright (X,\mu)$ be a non-singular action on a measure space. Observe that the action $\Gamma\curvearrowright L^\infty X$ by precomposition given by $\gamma \cdot f(x)=f(\gamma^{-1}\cdot x)$ is an action by $\ast$-homomorphisms; we denote by $L^\infty X\rtimesalg\Gamma$ the associated algebraic crossed product.
For $f\in L^\infty X$, let $M_f\in  B(L^2X)$ denote the multiplication operator with symbol $f$. Then $M_\bullet \colon L^\infty X \to  B(L^2 X), \; f \mapsto M_f$ is a $*$-representation. The cocycle property \eqref{eq:cocycle} of the Radon--Nikodym derivative shows that $(M_\bullet, \pi_\alpha)$ is a covariant representation of the pair $(L^\infty X, \Gamma)$, where $\pi_\alpha$ is the Koopman representation (see \eqref{eq:koopman}). Namely,
\[
    M_{\gamma\cdot f} = \pi_\alpha(\gamma)M_f\pi_\alpha(\gamma^{-1}).
\]
The induced $\ast$\=/representation of the algebraic crossed product is the $\ast$\=/homomorphism
\[
\begin{tikzcd}[row sep=0]
    \varPhi\colon &[-3em]  L^\infty X\rtimesalg\Gamma \arrow[r] & B(L^2X) \\
                &(f,\gamma) \arrow[|->,r] & M_f\pi_\alpha(\gamma) 
\end{tikzcd}
\]
discussed in the introduction.

Finally, we recall that if $\Gamma$ is acting on a $C^*$-algebra $A$, the \emph{full crossed product} $A\rtimesmax \Gamma$ is defined as the completion of $A\rtimesalg \Gamma$ with respect to the norm
\[
  \norm{x}_{\rm max}\coloneqq \sup\braces{\norm{\pi(x)}\mid \pi\text{ non-degenerate covariant representation}}.
\]
This is a $C^*$-algebra with the universal property that every $\ast$-homomorphism $A\rtimesalg\Gamma \to C$ into some $C^*$-algebra $C$ uniquely extends to a $\ast$-homomorphism $A\rtimesmax\Gamma \to C$.

\section{PROOF OF THEOREM~\ref{thmA:finite propagation as crossed product}}\label{sec:main thm A}
Fix a non-singular action $\alpha\colon\Gamma\curvearrowright (X,\mu)$. (As specified in Section \ref{ssec:actions}, $X$ is a standard Borel space, $\mu$ is a $\sigma$\=/finite measure, and $\Gamma$ is a discrete countable group.) We need to show that $\varPhi\colon L^\infty X \rtimesalg \Gamma \to \fpstar{\Gamma\curvearrowright X}$ is surjective.

For technical reasons, it is convenient to reduce to the case of probability measures $\mu$. This can easily be done using the naturality of the Koopman representation. More precisely, if $\nu\sim\mu$ is an equivalent measure, we already observed (see Remark \ref{rmk:algebras invariant of measureclass}) that $\Ad(U_{\mu,\nu})$ gives a $\ast$\=/isomorphism between $\fpstar{\Gamma\curvearrowright(X,\mu)}$ and $\fpstar{\Gamma\curvearrowright (X,\nu)}$, where $U_{\mu,\nu}\colon L^2(X,\mu)\to L^2(X,\nu)$ is the unitary operator defined by pointwise multiplication by $(\frac{d\mu}{d\nu})^{\frac 12}$.
Note that $\Ad(U_{\mu,\nu})$ acts as the identity on $L^\infty X$, seen once as a $\ast$\=/subalgebra of $ B(L^2(X,\mu))$ and once of $ B(L^2(X,\nu))$.

Let $\pi_{\alpha,\mu}$ and $\pi_{\alpha, \nu}$ be the Koopman representations defined by the action $\alpha$ with respect to the measures $\mu$ and $\nu$, respectively, and let $\varPhi_\mu$ and $\varPhi_\nu$ be the induced $\ast$\=/homomorphisms of $L^\infty X\rtimesalg\Gamma$. It follows from  \eqref{eq:cocycle} that $\pi_{\alpha,\nu}=\Ad(U_{\mu,\nu})\circ\pi_{\alpha,\mu}$. As a consequence, the following diagram commutes:
\[
 \begin{tikzcd}
  & \fpstar{\Gamma\curvearrowright (X,\mu)} \arrow[dd, "\Ad(U_{\mu,\nu})","\cong"'] \phantom{.}\\
  L^\infty X\rtimesalg\Gamma \arrow[ur, "\varPhi_\mu"] \arrow[dr, "\varPhi_\nu",swap] & \\
  & \fpstar{\Gamma\curvearrowright (X,\nu)}.
 \end{tikzcd}
\]

This proves the following lemma.

\begin{lem}
 If $\mu$ and $\nu$ are equivalent $\sigma$-finite measures, then $\varPhi_\mu$ is surjective (resp.~a $\ast$\=/isomorphism) if and only if $\varPhi_\nu$ is surjective (resp.~a $\ast$\=/isomorphism).
\end{lem}

Every $\sigma$-finite measure is equivalent to a probability measure. Thus, upon passing to a probability measure $\nu$ equivalent to $\mu$, it suffices to prove Theorem~\ref{thmA:finite propagation as crossed product} for probability measures. Therefore, for the rest of this section, we will \textbf{assume that $\mu$ is a probability measure}. We first establish two lemmas.

\begin{notation}
    Given $S\subseteq \Gamma$, we set $S_\times\coloneqq S\smallsetminus \{e\}$.
\end{notation}

\begin{lemma}\label{lem:remove e from support}
Let $T\in\fpstar{\Gamma\curvearrowright X}$, and let $S\subseteq\Gamma$ be a finite set such that $T$ is supported on $S$. Then there exists $f\in L^\infty X$ such that $T-M_f$ is supported on $S_\times$.
\end{lemma}
\begin{proof}
Let $\mathcal U$ denote the set of all countable measurable partitions of $X$, i.e.~partitions $(U_i)_{i\in I}$ such that $I$ is countable and each $U_i$ is measurable. A partition $(V_j)_{j\in J}$ is said to be subordinate to $(U_i)_{i\in I}$ if there is a function $t\colon J\to I$ such that $V_j\subseteq U_{t(j)}$. When this is the case, we write $(V_j)_{j\in J} \succeq (U_i)_{i\in I}$.
This relation defines a partial order on $\mathcal U$. Considering the intersection $(V_j)_{j\in J} \wedge (V'_{j'})_{j'\in J'} := (V_j\cap V'_{j'})_{(j,j')\in J\times J'}$ shows that $(\CU,\preceq)$ is a directed set.

Fix a cofinal ultrafilter on $\mathcal U$.
Let $A\subseteq X$ be measurable and consider the map $\mathcal U \to\C$ defined by
\begin{align*}
    (U_i)_{i\in I}&\mapsto\sum_{i\in I}\langle\mathbbm 1_{U_i\cap A},T\mathbbm 1_{U_i\cap A}\rangle.
\end{align*}
Note that $\abs{\langle\mathbbm 1_{U_i\cap A},T\mathbbm 1_{U_i\cap A}\rangle}\leq\norm T\mu(U_i\cap A)$. Therefore the above sum converges, and the absolute value of the sum is at most $\norm T\mu(A)$.
Hence we can use the cofinal ultrafilter to define the ultralimit
\[
\nu(A)\coloneqq \lim_{(U_i)_{i\in I}\in\mathcal U}\sum_{i\in I}\langle\mathbbm 1_{U_i\cap A},T\mathbbm 1_{U_i\cap A}\rangle.
\]

We claim that the function $\nu$ is $\sigma$-additive: Indeed, if $(A_k)_{k\in K}$ is some countable collection of pairwise disjoint measurable sets, then $\nu(\bigcup_{k\in K}A_k)=\sum_{k\in K}\nu(A_k)$.
This can be seen by taking the ultralimit along the set of measurable partitions of $X$ that are subordinate to the partition $(A_k)_{k\in K}\cup\{X\smallsetminus\bigcup_{k\in K}A_k\}$. This shows that $\nu\colon \{\text{Borel sets}\}\to \C$ defines a complex-valued measure.

For all measurable $A$ with $\mu(A)=0$ we also have $\nu(A)=0$. We may then define $f$ to be the Radon-Nikodym derivative $f=\frac{d\nu}{d\mu}$. For every measurable set $A$ we obtain
\[\lim_{(U_i)_{i\in I}\in\mathcal U}\sum_{i\in I}\langle\mathbbm 1_{U_i\cap A},(T-M_f)\mathbbm 1_{U_i\cap A}\rangle=0.\]

Now let $A\subseteq X$ be measurable and $\eta\in L^2A$.
We know that $(T-M_f)\eta$ is supported on $S\cdot A$, and we will show that it is supported on $S_\times \cdot A$.
For this, it is sufficient to show that $\langle\xi,(T-M_f)\eta\rangle=0$ for all $\xi\in L^2X$ supported on $A\smallsetminus\bigparen{S_\times \cdot A}$. By density, it in turn suffices to show this for simple functions $\xi$ and $\eta$.

Let $\xi$ and $\eta$ be simple functions.
Then the map $\xi\times\eta:X\to\C^2$ has finite range $R$.
Let $(U_i)_{i\in I}$ be a countable measurable partition subordinate to $((\xi\times\eta)^{-1}(r))_{r\in R}$.
Then we have
\[\sum_{i\in I}\langle\xi_{|U_i},(T-M_f)\eta_{|U_i}\rangle=\sum_{r=(r_1,r_2)\in R}\bar r_1r_2\sum_{U_i\subseteq(\xi\times\eta)^{-1}(r)}\langle\mathbbm1_{U_i},(T-M_f)\mathbbm1_{U_i}\rangle.\]
Taking the limit gives
\[\lim_{(U_i)_{i\in I}\in\mathcal U}\sum_{i\in I}\langle\xi_{|U_i},(T-M_f)\eta_{|U_i}\rangle=0.\]

On the other hand, note that $(T-M_f)\eta_{|U_i}$ is supported on $S\cdot U_i$, hence
\[
 \langle\xi,(T-M_f)\eta_{|U_i}\rangle=\langle\xi_{|S\cdot U_i},(T-M_f)\eta_{|U_i}\rangle.
\]
Moreover, since $\xi$ is supported on $A\smallsetminus\bigpar{S_\times \cdot A}$, if $U_i$ is a subset of $A$ then the above is equal to $\langle\xi_{|U_i},(T-M_f)\eta_{|U_i}\rangle$.
Since $\eta$ is supported in $A$, we conclude that for every partition $(U_i)_{i\in I}$ subordinate to $((\xi\times\eta)^{-1}(r))_{r\in R}$, we have
\[\langle\xi,(T-M_f)\eta\rangle
=\sum_{\substack{i\in I \\ U_i\subseteq A}}\langle\xi,(T-M_f)\eta_{|U_i}\rangle=\sum_{i\in I}\langle\xi_{|U_i},(T-M_f)\eta_{|U_i}\rangle.\]
Taking the limit gives $\langle\xi,(T-M_f)\eta\rangle=0$ by our choice of $f$.
\end{proof}

\begin{lemma}\label{lem:essentially-free-action-countable-cover}
Let $g\colon X\to X$ be a measurable map such that $g(x) \neq x$ for $\mu$-almost all $x\in X$.
Then there are countably many measurable sets $U_n\subseteq X$ such that the sets $U_n\smallsetminus g^{-1}(U_n)$ cover $X$ up to $\mu$-measure zero.
\end{lemma}
\begin{proof}
Since $X$ is a standard Borel space, we may choose a complete separable metric $d$ on it in such a way that the $\sigma$-algebra on $X$ equals the Borel $\sigma$-algebra of $d$.
For $\epsilon>0$, let $A_\epsilon = \{x\in X\mid d(g(x),x)<\epsilon\}$. These sets are nested, measurable, and their intersection has $\mu$-measure zero by hypothesis.

For every $N\in\NN$, we may choose $\varepsilon$ small enough so that $\mu(A_\epsilon)<\frac 1N$. By separability, we may cover $X$ with countably many measurable sets $U_k^N$ (indexed by $k$) of diameter at most $\epsilon$.
Then we have
\[
 \bigcup_k \Big(U_k^N \cap g^{-1}(U_k^N)\Big)  \subseteq A_\epsilon.
\]
Since
\[
 X=\bigcup_k U_k^N = \Big(\bigcup_k U_k^N \smallsetminus g^{-1}(U_k^N)\Big)\cup\Big(\bigcup_k U_k^N \cap g^{-1}(U_k^N)\Big),
\]
the sets $U_k^N\smallsetminus g^{-1}(U_k^N)$ cover $X$ up to measure $1/N$. Doing this for all $N$ and taking all the (still countably many) $U_k^N$ together, we see that $U_k^N\smallsetminus g^{-1}(U_k^N)$ cover $X$ up to a set of measure zero.
\end{proof}

\begin{proof}[Proof of Theorem \ref{thmA:finite propagation as crossed product}]
To prove surjectivity, fix $T\in\fpstar{\Gamma\curvearrowright X}$. Let $S\subseteq\Gamma$ be finite such that $T$ is supported on $S$.
We show that $T$ is in the image of $\varPhi$ by induction on the size of $S$.

If $S$ is empty, we have $T=0$ and we are done. Now, suppose that $S$ is not empty and pick $s\in S$. For every $\eta$ supported in $A$, the function $\pi_\alpha(s^{-1})T\eta$ is supported on $s^{-1}\cdot (S\cdot A)= (s^{-1}S)\cdot A$. Applying Lemma~\ref{lem:remove e from support} to the operator $\pi_\alpha\paren{s^{-1}}T$, we obtain an $f\in L^\infty X$ such that for every measurable set $A$ and $\eta\in L^2A$, the function $(\pi_\alpha(s^{-1})T-M_f)\eta$ is supported on $(s^{-1}S\smallsetminus\{e\})\cdot A$.
Then $(T-\pi_\alpha(s)M_f)\eta$ is supported on $(S\smallsetminus\{s\})\cdot A$, so $T-\pi_\alpha(s)M_f$ is supported on $S\smallsetminus\{s\}$.
By the induction hypothesis, $T-\pi_\alpha(s)M_f$ is in the image of $\varPhi$, and so is $\pi_\alpha(s)M_f=M_{s^{-1}f}\pi_\alpha(s)$. Hence, $T$ is also in the image of $\varPhi$, completing the induction argument.

Now suppose that the action is essentially free and pick $T=\sum_{s\in S}(f_s,s)$ in the kernel of $\varPhi$, with $S\subseteq\Gamma$ finite. By Lemma \ref{lem:essentially-free-action-countable-cover}, for each $s\in S_\times $, there is a countable collection $(U_{n}^{(s)})$ such that the $U_{n}^{(s)}\smallsetminus s\cdot U_{n}^{(s)}$ cover $X$ up to a set of measure zero. 
Since $S$ is finite, the set $I=\NN^{S_\times}$ is countable. For $i\in I$, let 
\[
    U_i \coloneqq \bigcap_{s\in S_\times }U_{i(s)}^{(s)}. 
\]
We claim that the countably many sets $U_i\smallsetminus\bigpar{S_\times \cdot U_i}$ cover $X$ up to a set of measure zero.

In fact, for a.e.\ $x\in X$ there is an $i\in I$ such that $x\in U_{i(s)}^{(s)}\smallsetminus s\cdot U_{i(s)}^{(s)}$ for every $s\in S_\times$. In particular, $x\in U_i$. Then we are done because
\[
    S_\times \cdot U_i\subseteq \bigcup_{s\in S_\times} s\cdot U_{i(s)}^{(s)}
\]
does not contain $x$.

Now for every $\xi\in L^2X$ supported on $U_i\smallsetminus\bigpar{S_\times \cdot U_i}$, we have
\[
0=\langle\xi,\varPhi(T)\mathbbm 1_{U_i}\rangle
=\langle\xi,\sum_{s\in S}f_s\pi_\alpha(s)(\mathbbm 1_{U_i})\rangle
=\langle\xi,f_e\pi_\alpha(e)(\mathbbm 1_{U_i})\rangle
=\langle\xi,f_e\mathbbm 1_{U_i}\rangle.
\]
Hence, $f_e$ is $0$ on $U_i\smallsetminus\bigparen{S_\times \cdot U_i}$ for all $i$. Therefore, $f_e=0$.

For every $s\in S$, we also have that $\pi_\alpha\paren{s^{-1}}T=\sum_{s'\in S}(s^{-1}\cdot f_{s'},s^{-1}s')$ is an element of the kernel of $\varPhi$, so $s^{-1}\cdot f_s=0$.
This shows that $T$ is the trivial element in $L^\infty X\rtimesalg\Gamma$, i.e.~$\varPhi$ is injective and hence a $\ast$\=/isomorphism.
\end{proof}

Now that Theorem \ref{thmA:finite propagation as crossed product} has been proved, the following simple consequence is all we need to deduce the operator-algebraic characterizations of ergodicity and strong ergodicity stated as Corollary \ref{corA:ergodicity} and \ref{corA:strong ergodicity}.

\begin{prop}\label{prop:commutant of fpstar}
The commutant $\fpstar{\Gamma\curvearrowright X}'$ coincides with the space $(L^\infty X)^\Gamma$ of $\Gamma$\=/invariant bounded functions.
\end{prop}
\begin{proof}
 Since $\fpstar{\Gamma\curvearrowright X}=\Gamma\cdot L^\infty X$, it is clear that $(L^\infty X)^\Gamma$ is contained in the commutant.

 For the converse containment, note that the commutant of $\fpstar{\Gamma\curvearrowright X}$ is a $\ast$\=/subalgebra of $L^\infty X$, because $(L^\infty X)'=L^\infty X$ (see e.g.~\cite[Section 4]{murphy2014c}). Let $f\in L^\infty X$ be in $\fpstar{\Gamma\curvearrowright X}'$. Then $M_f$ must commute with $\pi_\alpha(\gamma)$ for every $\gamma\in \Gamma$. This in turn implies that
 \[
  M_f=\pi_\alpha(\gamma)M_f\pi_\alpha(\gamma)^{-1}=M_{\gamma \cdot f},
 \]
 and hence $f=\gamma\cdot f$.
\end{proof}

\section{PRELIMINARIES: ROE ALGEBRAS AND WARPED SPACES}\label{sec:prelims of roe algebras and warped metrics}

\subsection{Roe algebras} \label{ssec:roe algebras}
We recall some definitions and facts regarding Roe algebras of metric spaces, many of which essentially go back to \cite{RoeIndexI,RoeIndexII,RoeMAMS,Roe1995}. We refer to \cite{willett2020higher} for details and proofs.

Let $(Y,d)$ be a proper metric space and $C_0(Y)$ the $C^*$-algebra of complex-valued continuous functions vanishing at infinity.

\begin{definition}
A \emph{(geometric) $Y$-module} is a separable Hilbert space $\H_Y$ together with a non\=/degenerate $*$\=/representation $\rho\colon C_0(Y)\to  B(\H_Y)$.
A $Y$\=/module is \emph{ample} if every non\=/zero $f\in C_0(Y)$ is mapped to a non\=/compact operator.
\end{definition}

Let $\H_Y$ be a $Y$-module. An operator $T\in B(\H_Y)$ has \emph{finite propagation} if there exists $R\geq 0$ such that for all $f,g\in C_0(Y)$ with $d(\supp(f),\supp(g)) >R$, we have $\rho(f) T\rho(g)=0$. It can be shown that finite propagation is preserved under addition and composition of operators. It follows that the set of operators of finite propagation is a $*$-subalgebra of $ B(\H_Y)$, which we denote by $\fpstar{\CH_Y}$.
The \emph{$C^*$-algebra of operators with finite propagation} is the norm closure $\fpcstar{\CH_Y}$ of $\fpstar{\CH_Y}$ in $ B(\H_Y)$.

\begin{rmk}\label{rmk:alpha_propagation}
Every countable discrete group can be given a proper left-invariant metric $d_\Gamma$ (such a metric is uniquely defined up to coarse equivalence). An action $\alpha\colon\Gamma\curvearrowright X$ can be used to define an extended metric $d_{\rm orb}$ on $X$ 
declaring the distance between two points in the same $\Gamma$-orbit to be the shortest length of a group element mapping one to the other. The distance between disjoint orbits is set to be $+\infty$. An operator $T\in B(L^2X)$ has finite dynamical propagation if and only if it has finite propagation with respect to the extended metric $d_{\rm orb}$. This justifies our choice of notation for $\fpstar{\Gamma\curvearrowright X}$. Note, however, that it is somewhat improper to talk about the geometric propagation of $T$, because $L^2X$ is not a geometric module for the extended metric space $Y\coloneqq (X,d_{\rm orb})$: 
The algebra $C_0(Y)$ does not have a good $\ast$-representation on $L^2X$ (a compactly supported continuous function is supported on finitely many $\Gamma$\=/orbits, so pointwise multiplication yields the $0$ representation whenever $\mu$ is non-atomic).
\end{rmk}

For a topological space $Y$, let $C_c(Y)$ be the algebra of compactly supported continuous functions on $Y$. An operator $T\in B(\H_Y)$ is \emph{locally compact} if for every $f\in C_c(Y)$ the operators $\rho(f)T$ and $T\rho(f)$ are compact. The set $C_{lc}^*(\CH_Y)$ of (bounded) locally compact operators is a $C^*$-subalgebra of $B(\H_Y)$.

\begin{definition}
 The \emph{Roe $*$-algebra $\roestar{\H_Y}$} of a $Y$-module $\H_Y$ is the $*$-algebra of all locally compact operators with finite propagation:
 \[
    \roestar{\H_Y} = \fpstar{\H_Y}\cap C_{lc}^*(\CH_Y).
 \]
 The \emph{Roe algebra} of $\H_Y$ is defined as the closure $\roecstar{\CH_Y}\coloneqq\overline{\roestar{\H_Y}}^{\,\norm{\cdot}}$ in the norm topology.
 The \emph{Roe algebra $\roecstar{Y}$} of the metric space $Y$ is the Roe algebra of any of its ample geometric modules; up to $\ast$-isomorphism, $\roecstar{Y}$ does not depend on the choice of ample module.
\end{definition}

\begin{lem}\label{lem:roe alg is ideal}
 The Roe $*$-algebra $\roestar{\CH_Y}$ is a two-sided ideal in $\fpstar{\CH_Y}$.
\end{lem}
\begin{proof}
 It suffices to show that if $T\in B(\CH_Y)$ is locally compact and $V\in\fpstar{\CH_Y}$, then $TV$ and $VT$ are locally compact. Fix any $f\in C_c(Y)$. The operator $\rho(f)TV$ is the composition of a bounded and a compact operator, hence compact. To show that $TV\rho(f)$ is compact, we use that $V$ has finite propagation (and $Y$ is a proper metric space) to deduce that there is a compact $Q\subseteq Y$ so that $\rho(g)V\rho(f)=0$ whenever $g$ is supported on the complement of $Q$. Let $h\in C_c(Y)$ be so that $h$ is identically $1$ on $Q$. Then $\rho(h)V\rho(f)=V\rho(f)$, and hence
 $TV\rho(f)= T\rho(h)V\rho(f)$ is again the composition of a bounded and a compact operator. An analogous argument applies to $VT$ as well.
\end{proof}

\begin{cor}
 The Roe algebra $\roecstar{\CH_Y}$ is a closed two-sided ideal in $\fpcstar{\CH_Y}$.
\end{cor}

One useful feature of geometric modules is that they extend nicely to measurable functions. Indeed, every non-degenerate $\ast$-representation $\rho \colon C_0(Y)\to B(\H_Y)$ has a canonical extension $\rho \colon B_b(Y)\to B(\H_Y)$, where $B_b(Y)$ is the $C^*$\=/algebra of bounded Borel functions on $Y$. In particular, for every measurable set $A\subseteq Y$, the indicator function $\mathbbm 1_{A}\in B_b(Y)$ is sent to a projection $\chi_{A}\in B(\H_Y)$. Furthermore, if $(A_i)_{i\in \NN}$ is a measurable partition of $Y$, then the finite sums $\sum_{i=1}^N\chi_{A_i}\in B(\H_Y)$ converge in the strong operator topology to the identity operator $1\in B(\H_Y)$ as $N$ tends to infinity \cite[Proposition 1.6.11]{willett2020higher}.
\\

Roe algebras are particularly well behaved and most interesting for metric spaces with bounded (coarse) geometry. A metric space $(Y,d)$ has \emph{bounded geometry} if there exists a \emph{gauge} $\varrho>0$ such that for every $R > 0$ there is $N(R) \in \NN$ such that every ball of radius $R$ in $Y$ is contained in a union of $N(R)$ balls of radius $\varrho$. The following lemma will be needed later.

\begin{lem}\label{lem:partitioning bded geometry spaces} 
 Let $(Y,d)$ be a metric space with bounded geometry, and let $\varrho>0$ be a gauge realizing the bounded geometry condition. For every $R>0$, there is a finite partition $Y=\bigsqcup_{i=1}^nY_i$ and further subpartitions $Y_i=\bigsqcup_{z\in Z_i} U_{z,i}$ such that
 \begin{itemize}
  \item $Z_i$ is countable for every $i$;
  \item each $U_{z,i}$ is a Borel set and $\diam(U_{z,i})\leq 4\varrho$;
  \item $d(U_{z,i},U_{z',i})>R$ for every $z,z'\in Z_i$ with $z \neq z'$.
 \end{itemize}
\end{lem}
\begin{proof}
Choose a maximal $2\varrho$-separated subset $Z\subseteq Y$. Observe that as $z\in Z$ varies, the (open) balls $B_Y(z;2\varrho)$ cover $Y$. Since $Z$ is $2\varrho$-separated, any two $z\neq z'\in Z$ cannot belong to a common ball of radius $\varrho$. It then follows from the bounded geometry assumption that $Z$ is countable, so we may enumerate it: $Z=\{z_0,z_1,\ldots\}$. Let $U_{z_k}\coloneqq B_Y(z_k;2\varrho)\smallsetminus(\bigcup_{j<k}U_{z_j})$. Note that this defines a measurable partition of $Y$.

Consider the graph with vertices $Z$ and an edge between $z_j$ and $z_k$ if and only if $d(U_{z_j},U_{z_k})\leq R$. The bounded geometry condition implies that this graph has bounded degree, say $D$. We may now choose a $(D+1)$-coloring of $Z$ in such a way that no two vertices of the same colour are joined by an edge. Taking $I$ to be the set of colors, $Z_i\subseteq Z$ the set of vertices of color $i$, and $U_{z, i}\coloneqq U_z$ for every $z\in Z_i$, the result follows.
\end{proof}

Analogous to the setting of maximal crossed products (or maximal group $C^*$-algebras), one can define (on spaces with bounded geometry) a maximal version of the Roe algebra \cite{gong2008geometrization} (see also \cite{winkel2021geometric}).
Namely, if $(Y,d)$ has bounded geometry and $\CH_Y$ is an ample $Y$-module, then
\[
    \norm{x}_{\rm max}\coloneqq 
    \sup\braces{\norm{\pi(x)}\mid \pi \colon\roestar{\CH_Y}\to B(\CH)\text{ $*$-representation}}
\]
defines a $C^\ast$-norm on $\roestar{\CH_Y}$.

\begin{de}\label{def:maximal Roe algebra}
    The \emph{maximal Roe algebra} of the space $Y$ is the $C^*$-algebra $\roecstarmax{Y}$ defined as the completion of $\roestar{\CH_Y}$ with respect to $\norm{\cdot}_{\rm max}$.
\end{de}

By construction, $\roecstarmax{Y}$ has the universal property that every $\ast$\=/homomorphism $\roestar{\CH_Y}\to A$  to a $C^*$\=/algebra $A$ lifts to a unique $\ast$\=/homomorphism of  $\roecstarmax{Y}$.

\subsection{Warped spaces} \label{ssec:warped metric}
A \emph{length function} $\ell$ on a countable group $\Gamma$ is a function $\ell\colon\Gamma\to[0,\infty)$ that is symmetric and subadditive satisfying $\ell(\gamma)=0$ if and only if $\gamma=e$. It is \emph{proper} if for every $r \geq 0$, there are only finitely many $\gamma$ of length $\leq r$. There is a one-to-one correspondence between (proper) length functions and (proper) left-invariant metrics on $\Gamma$. Every countable group admits proper length functions, and they all give rise to coarsely equivalent metrics.

Let $\alpha \colon \Gamma\curvearrowright (Y,d)$ be an action, and let $\ell$ be a proper length function on $\Gamma$. By rescaling, we may always assume that $\min\{\ell(\gamma)\mid \gamma\in \Gamma \smallsetminus \{e\}\}=1$. We denote by $B_\Gamma(r)\subseteq\Gamma$ the set of elements of length $<r$.

\begin{de}
The \emph{warped metric} $\delta_\Gamma$ on $Y$ is the largest metric such that
\[
\delta_\Gamma \leq d \quad
\text{ and } \quad
\delta_\Gamma\paren{y,\gamma\cdot y}\leq \ell(\gamma)
\]
for every $y\in Y$ and $\gamma\in \Gamma$.
\end{de}

It is simple to verify that $\delta_\Gamma$ is given by the formula
\begin{equation}\label{eq:prelims:warped.distance}
  \delta_\Gamma(x,y) =\inf\left\{\sum_{i=1}^n\ell(\gamma_i)+\sum_{i=0}^n d(x_i,y_i)\right\},  
\end{equation}
where the infimum is taken over $n\in\NN$, $(n+1)$\=/tuples $x_0,\ldots,x_n$ and $y_0,\ldots,y_n$ and $n$\=/tuples $\gamma_1,\ldots,\gamma_n\in\Gamma$ such that $x=x_0$, $y=y_n$ and $x_{i}=\gamma_i (y_{i-1})$ for every $i=1,\ldots,n$.

It is straightforward to verify that $d$ and $\delta_\Gamma$ generate the same topology on $Y$. It follows that geometric $(Y,d)$-modules and geometric $(Y,\delta_\Gamma)$-modules coincide (but the algebras of operators of finite propagation do not!).

For a subset $A$ in a metric space $Z$, let $N_Z(A;r)\coloneqq\{z\in Z\mid d(z,A)<r\}$ denote the $r$-neighbourhood of $A$. In the sequel we will need the following result. Its proof is easy for actions by isometries. The non-isometric case is a little more involved, but it is essentially contained in \cite[Proposition 1.7]{roe2005warped} (see also \cite[Lemma 4.2]{li2020markovian}).

\begin{lem}\label{lem:balls in finite products}
 If $\Gamma\curvearrowright (Y,d)$ is an action by Lipschitz maps, then for every $r > 0$, there exists $R > 0$ such that
 \[
 N_{(Y,\delta_\Gamma)}\paren{A;r}\subseteq B_\Gamma(r)\cdot N_{(Y,d)}(A; R)
 \]
 for every $A \subseteq Y$.
\end{lem}
\begin{proof}[Sketch of proof]
 The proof is by induction on $\lfloor r\rfloor$. We may assume that $\ell(\gamma)\geq 1$ for every $\gamma\in \Gamma \smallsetminus \{e\}$. It then follows that if $\lfloor r\rfloor= 0$, the warping condition does not have any effect; hence $N_{(Y,\delta_\Gamma)}\paren{A;r}\subseteq N_{(Y,d)}(A; r)$.

 For the inductive step, with the help of \eqref{eq:prelims:warped.distance} it is straightforward to verify that $ N_{(Y,\delta_\Gamma)}(A; r)$ is contained in the union
 \begin{equation*}\label{eq:neighbourhood.containment}
  N_{(Y,d)}(A; r)
  \cup \left( \bigcup_{1\leq m < r}
  N_{(Y,\delta_\Gamma)}\Bigparen{\overline B_\Gamma(m)\cdot N_{(Y,d)}(A;r-m)\, ;\, r-m}
  \right),
 \end{equation*}
 where $\overline B_\Gamma(m)$ denotes the closed ball of radius $m$. By induction, there is an $R' > 0$ such that the above is contained in
 \[
  \bigcup_{1\leq m < r}
  B_\Gamma(r-m) \cdot N_{(Y,d)}\Bigparen{\overline B_\Gamma(m)\cdot N_{(Y,d)}(A;r-m)\, ;\, R'}
 \]
 and hence in
 \[
  B_\Gamma(r) \cdot N_{(Y,d)}(A;LR'+r),
 \]
 where $L$ is the maximum of the Lipschitz constants over all $\gamma\in B_\Gamma(r)$.
\end{proof}

\begin{cor} \label{cor:warped bounded geometry}
 If $\Gamma\curvearrowright (Y,d)$ is a Lipschitz action and $(Y,d)$ has bounded geometry, then $(Y,\delta_\Gamma)$ has bounded geometry as well.
\end{cor}

\begin{rmk}\label{rmk:action by coarse equivalence suffices}
    Lemma~\ref{lem:balls in finite products} and Corollary~\ref{cor:warped bounded geometry} also hold true---with the same proof---for actions by coarse equivalences.
\end{rmk}

\section{ROE ALGEBRAS OF (WARPED) METRIC SPACES} \label{sec:roe alg warped spaces}

In this section, we describe the relation between algebras of operators of finite dynamical propagation and ``Roe-like'' algebras of (warped) spaces. We work under the following hypothesis.

\begin{assumption} \label{ass:Lipschitz action on proper}
 $(Y,d,\nu)$ is a proper metric space with a $\sigma$-finite Borel measure and $\alpha\colon \Gamma\curvearrowright Y$ is a non-singular Lipschitz action of a countable discrete group.
\end{assumption}

We fix $L^2Y$ as a geometric module. Rather than writing $\roestar{L^2Y}$ for the Roe $\ast$\=/algebra, in this section we keep the dependence on the metric more explicit by writing $\roestar{Y,d}$ instead. We do the same for the other ``Roe-like'' algebras. Note that there is some notational inconsistency here: With this convention, the $C^*$-completion $\roecstar{L^2Y}$ becomes $\roecstar{Y,d}$, which should denote the Roe algebra of the metric space $(Y,d)$. However, this is only the case if $L^2Y$ is an ample $Y$-module. When discussing this algebra, it is thus better to be unambiguous by assuming that this is indeed the case. Hence, when we consider the Roe algebra, we also make the following assumption.

\begin{assumption} \label{ass:ample module}
 $L^2Y$ is an ample $(Y,d)$-module. (This is for example the case if $\nu$ is non-atomic and has full support).
\end{assumption}

In our results below, we will explicitly refer to these assumptions when we make (one of) them. We begin with the following observation.

\begin{lem} \label{lem:the action normalizes}
Let $\alpha\colon \Gamma\curvearrowright Y$ be a group action satisfying Assumption~\ref{ass:Lipschitz action on proper}. Then $\pi_\alpha(\Gamma)$ normalizes $\fpstar{Y,d}$ and $C^*_{lc}(Y,d)$, and hence also $\roestar{Y,d}$.
\end{lem}

\begin{proof}
 Let $T$ be an operator with finite propagation, say $r>0$, and fix $\gamma\in \Gamma$. Let $L$ denote the bi-Lipschitz constant of $\alpha_\gamma$. We claim that $\pi_\alpha(\gamma^{-1})T\pi_\alpha(\gamma^{-1})^*=\pi_\alpha(\gamma)^*T\pi_\alpha(\gamma)$ has propagation at most $Lr$. Let $f,g\in C_0(Y)$ be functions with $d(\supp(f),\supp(g))>Lr$. Note that $\supp(\gamma\cdot f)= \gamma \cdot \supp(f)$, and hence,
 \[
  d(\supp(\gamma\cdot f),\supp(\gamma \cdot g))>r.
 \]
 We therefore have
 \[
  M_f\pi_\alpha(\gamma)^*T\pi_\alpha(\gamma)M_g=
  \pi_\alpha(\gamma)^*M_{\gamma\cdot f}TM_{\gamma\cdot g}\pi_\alpha(\gamma) = 0,
 \]
 as desired.

 Let now $T\in C^*_{lc}(Y,d)$ be a locally compact operator. Fix $\gamma\in \Gamma$ and $f\in C_c(Y)$. Since the action is continuous, $\gamma\cdot f$ is still in $C_c(Y)$; hence
 \[
  \pi_\alpha(\gamma)^*T\pi_\alpha(\gamma) M_f=\pi_\alpha(\gamma)^* T M_{\gamma\cdot f} \pi_\alpha(\gamma)
 \]
 is compact as composition of a compact and bounded operators.
\end{proof}

\begin{cor}
Let $\alpha\colon \Gamma\curvearrowright Y$ be a group action satisfying Assumption~\ref{ass:Lipschitz action on proper}. Then the $*$-algebra $\fpstar{\Gamma\curvearrowright Y}$ normalizes $\fpstar{Y,d}$ and $C^*_{lc}(Y,d)$, and hence also $\roestar{Y,d}$.
\end{cor}
\begin{proof}
 By Theorem~\ref{thmA:finite propagation as crossed product}, $\fpstar{\Gamma\curvearrowright Y}$ is the surjective image of $L^\infty Y\rtimesalg\Gamma$, so by Lemma \ref{lem:the action normalizes} it remains to verify that $L^\infty Y$ normalizes the algebras $\fpstar{Y,d}$ and $C^*_{lc}(Y,d)$.

 Since $L^\infty Y$ is contained in $\fpstar{Y,d}$, it clearly normalizes it. Also, since $L^\infty Y$ commutes with $C_c(Y)$, it is easy to see that multiplying by functions in $L^\infty Y$ preserves local compactness.
\end{proof}

\begin{cor}
Let $\alpha\colon \Gamma\curvearrowright Y$ be a group action satisfying Assumptions~\ref{ass:Lipschitz action on proper} and \ref{ass:ample module}. Then $\fpcstar{\Gamma\curvearrowright Y}$ normalizes $\fpcstar{Y,d}$ and $\roecstar{Y,d}$, where the latter is the Roe algebra.
\end{cor}

The following result follows in a straightforward way from Theorem~\ref{thmA:finite propagation as crossed product}.

\begin{prop}\label{prop:equalities of algebras}
Under Assumption~\ref{ass:Lipschitz action on proper}, we have the following equalities of $\ast$-subalgebras of $ B(L^2Y)$:
 \begin{align*}
  \Gamma\cdot \fpstar{Y,d} &=
  \fpstar{\Gamma\curvearrowright Y}\cdot \fpstar{Y,d},
  \\
  \Gamma\cdot \roestar{Y,d} &=
  \fpstar{\Gamma\curvearrowright Y}\cdot \roestar{Y,d},
 \end{align*}
 Taking the closure, under Assumption~\ref{ass:ample module} (so that $\roecstar{Y,d}$ is the Roe algebra), we also have
 \begin{align*}
  \overline{\Gamma\cdot \fpcstar{Y,d}}^{\norm{\cdot}} &=
  \overline{\fpcstar{\Gamma\curvearrowright Y}\cdot \fpcstar{Y,d}}^{\norm{\cdot}},
  \\
  \overline{\Gamma\cdot \roecstar{Y,d}}^{\norm{\cdot}} &=
  \overline{\fpcstar{\Gamma\curvearrowright Y}\cdot \roecstar{Y,d}}^{\norm{\cdot}}.
 \end{align*}
\end{prop}
\begin{proof}
    The containment in one direction is clear, because $\fpstar{\Gamma\curvearrowright Y} = \Gamma\cdot L^\infty Y$ by Theorem~\ref{thmA:finite propagation as crossed product} and $L^\infty Y$ is unital. For the other containments, it is sufficient to observe that $L^\infty Y\circ \fpstar{Y,d} \subseteq \fpstar{Y,d}$ and $L^\infty Y\circ \roestar{Y,d} \subseteq \roestar{Y,d}$.
\end{proof}

\begin{rmk}
The $\ast$-algebra $\Gamma\cdot \fpstar{Y,d}$ contains both $\fpstar{\Gamma\curvearrowright Y}$ and $\fpstar{Y,d}$ as $\ast$\=/subalgebras. In other words, $\Gamma\cdot \fpstar{Y,d}$ is the $\ast$\=/subalgebra of $ B(L^2Y)$ generated by $\fpstar{\Gamma\curvearrowright Y}$ and $\fpstar{Y,d}$.

However, while it is still true that $\Gamma\cdot \roestar{Y,d}$ contains $\roestar{Y,d}$, it will in general not contain $\fpstar{\Gamma\curvearrowright Y}$ (and not even $\pi_\alpha(\Gamma)$), for these operators are in general not locally compact.
\end{rmk}

\begin{rmk}
 Regarding the $C^{\ast}$-completions, we point out that the algebras $\fpcstar{\Gamma\curvearrowright Y}\cdot \fpcstar{Y,d}$ and $
  \fpcstar{\Gamma\curvearrowright Y}\cdot \roecstar{Y,d}$ are contained in the closures of $\Gamma\cdot \fpcstar{Y,d}$ and $\Gamma\cdot \roecstar{Y,d}$, respectively. However, a priori they need not be closed, so we are not sure whether this containment is strict or not.
\end{rmk}

We now prove that under the assumption of bounded geometry, the algebras we just constructed are the algebra of operators with finite propagation and the Roe algebra of the warped metric space $(Y,\delta_\Gamma)$, respectively. The key result is the following.

\begin{prop}\label{prop:span generates fpstar}
Let $\alpha\colon \Gamma\curvearrowright Y$ be a group action satisfying Assumption~\ref{ass:Lipschitz action on proper}. If $\paren{Y,d}$ has bounded geometry, then
 \(
 \fpstar{{Y,\delta_\Gamma}} = \Gamma\cdot \fpstar{{Y,d}}
 \)
\end{prop}

We first need to prove a lemma. Recall from Lemma~\ref{lem:balls in finite products} that in the setting of Proposition \ref{prop:span generates fpstar}, for every $r > 0$ there exists $R=R(r)>0$ with
 \[
 N_{\paren{Y,\delta_\Gamma}}\paren{A;r}\subseteq B_\Gamma(r)\cdot N_{\paren{Y,d}}(A; R)
 \]
for every $A\subseteq {Y}$.

\begin{lem}\label{lem:partitioning neighbourhoods}
Given a Lipschitz action $\Gamma\curvearrowright Y$, any subset $A\subseteq Y$, $r>0$ and $R(r)$ as above, there exists a measurable partition
\[
    N_{\paren{Y,\delta_\Gamma}}\paren{A;r}=\bigsqcup_{\gamma\in B_\Gamma(r)}\Omega_\gamma
\]
such that each $\gamma^{-1}(\Omega_\gamma)\subseteq N_{\paren{Y,d}}\paren{A;R}$ for every $\gamma\in B_\Gamma(r)$.
\end{lem}
\begin{proof}
 Let
 \[
  \Omega_{\gamma}'
  \coloneqq \paren{\gamma\cdot N_{\paren{Y,d}}\paren{A;R}} \cap N_{\paren{Y,\delta_\Gamma}}\paren{A;r}.
 \]
 Enumerate the elements in $B_\Gamma(r)$ as $\gamma_1,\ldots,\gamma_n$ and let $\Omega_{\gamma_i}\coloneqq \Omega_{\gamma_i}'\smallsetminus(\bigcup_{j<i}\Omega_{\gamma_j})$. Observe that these sets are Borel, pairwise disjoint, and they cover $N_{\paren{Y,\delta_\Gamma}}\paren{A;r}$ by Lemma~\ref{lem:balls in finite products}.
\end{proof}

\begin{proof}[Proof of Proposition~\ref{prop:span generates fpstar}]
One of the inclusions is rather easy and holds in general: For every non-singular action $\alpha\colon\Gamma\curvearrowright Y$, the image $\pi_\alpha(\Gamma)$ of the Koopman representation lies in $\fpstar{{Y,\delta_\Gamma}}$. In fact,
for every $\gamma\in\Gamma$, it follows immediately from the definition of the warped metric that $\pi_\alpha(\gamma)$ has propagation bounded by $\ell(\gamma)$.
It follows that under Assumption~\ref{ass:Lipschitz action on proper},
 \(
 \Gamma\cdot \fpstar{{Y,d}}\subseteq \fpstar{{Y,\delta_\Gamma}}.
 \)

It remains to prove the other inclusion, which is the most interesting one.
Let $T\in B(\CH_{{Y,d}})$ be an operator with propagation at most $r$ with respect to the warped metric $\delta_\Gamma$, let $R=R(r)$ be as in Lemma~\ref{lem:balls in finite products}, and let $\varrho>0$ be a gauge witnessing the bounded geometry of 
$\paren{Y,d}$.
 Apply Lemma~\ref{lem:partitioning bded geometry spaces} with respect to $\varrho$ and $2R$ to define a measurable partition ${Y}=\bigsqcup_{i=1}^n Y_i$ with
 \begin{itemize}
  \item $Y_i=\bigsqcup_{z\in Z_i} U_{z,i}$;
  \item $Z_i$ countable for every $i$;
  \item $U_{z,i}$ Borel with $\diam_d(U_{z,i})\leq 4\varrho$;
  \item $d(U_{z,i},U_{z',i})> 2R$ for every $z,z'\in Z_i$ with $z \neq z'$.
 \end{itemize}
Here we used $2R$ to make sure that
 \begin{equation}\label{eq: disjoint nbhds}
     N_{\paren{Y,d}}(U_{z,i}; R)\cap N_{\paren{Y,d}}(U_{z',i}; R) = \emptyset \qquad \forall z\neq z'\in Z_i.
 \end{equation}
 
 Let $T_i\coloneqq  T\chi_{Y_i}$. Since $T=\sum_{i=1}^n T_i$ is a finite sum, it is enough to show that $T_i\in \Gamma\cdot \fpstar{{Y,d}}$ for each $i$.
 
 Note that $T_i=\sum_{z\in Z_i}T\chi_{U_{z,i}}$, where the sum is defined by SOT convergence. 
 Since $T$ has propagation $\leq r$, we have
 $ T\chi_{U_{z,i}}=\chi_{N_{\paren{Y,\delta_\Gamma}}(U_{z,i}; r)}T\chi_{U_{z,i}}$.
 For every $z\in Z_i$, use Lemma~\ref{lem:partitioning neighbourhoods} to decompose $N_{\paren{Y,\delta_\Gamma}}(U_{z,i}; r)$ as $\bigsqcup_{\gamma\in B_\Gamma(r)}\Omega_{z,i,\gamma}$
 where 
 \begin{equation}\label{eq: preimage of Omegas}
     \gamma^{-1}(\Omega_{z,i,\gamma})\subseteq N_{\paren{Y,d}}(U_{z,i}; R)
 \end{equation}
 for every $\gamma\in B_\Gamma(r)$. Hence we obtain
 \begin{equation}\label{eq: T_i as sum}
 T_i=\sum_{z\in Z_i}\sum_{\gamma\in B_\Gamma(r)}\chi_{\Omega_{z,i,\gamma}}T\chi_{U_{z, i}}.
 \end{equation}

Now, fix $\gamma\in B_\Gamma(r)$. Observe that \eqref{eq: disjoint nbhds} and \eqref{eq: preimage of Omegas} imply that 
 \[
 \Omega_{z,i,\gamma}\cap \Omega_{z',i,\gamma}=\emptyset \qquad \forall z\neq z'\in Z_i.
 \]
 Therefore, as $z\in Z_i$ vary, the operators $\chi_{\Omega_{z,i,\gamma}}T\chi_{U_{z, i}}$ are all orthogonal to one another and have norm bounded by $\norm{T}$. It follows that for every $\gamma\in B_\Gamma(r)$
 \[
  T_{\gamma,i}\coloneqq\sum_{z\in Z_i}\chi_{\Omega_{z,i,\gamma}}T\chi_{U_{z, i}}
 \]
 is a well-defined bounded operator. Since $B_\Gamma(r)$ is finite, we can safely exchange the order of summation in \eqref{eq: T_i as sum} to see that
 \[
 T_i=\sum_{\gamma\in B_\Gamma(r)}\sum_{z\in Z_i}\chi_{\Omega_{z,i,\gamma}}T\chi_{U_{z, i}} = \sum_{\gamma\in B_\Gamma(r)} T_{\gamma,i}.
 \]

  Then composing with the Koopman representation we obtain
 \[
  \pi_\alpha(\gamma^{-1}) T_{\gamma,i}=
  \sum_{z\in Z_i}\pi_\alpha(\gamma^{-1})\chi_{\Omega_{z,i,\gamma}}T\chi_{U_{z, i}}=
  \sum_{z\in Z_i}\chi_{\gamma^{-1}(\Omega_{z,i,\gamma})}\pi_\alpha(\gamma^{-1})T\chi_{U_{z, i}}.
 \]
 Note that the operators in the above is a sum are mutually orthogonal and each of them has propagation at most $R+4\varrho$ in  $\paren{Y,d}$ by construction. Therefore $\pi_\alpha(\gamma^{-1}) T_{\gamma,i}\in\fpstar{{Y,d}}$ and $T_{\gamma,i}\in \fpstar{\Gamma\curvearrowright Y}\circ \fpstar{{Y,d}}$.
\end{proof}

    The operators $T_i$, $T_{\gamma,i}$ etc.~appearing in the proof of Proposition~\ref{prop:span generates fpstar} are all of the form $\chi_A T\chi_B$. Since the projection $\chi_A$ commutes with $\rho(f)$ for every bounded Borel function $f$, it follows that whenever $T$ is locally compact, then all these operators are locally compact too.

Knowing this, it is now straightforward to prove Theorem \ref{thmA:roe alg of warped spaces}.

\begin{proof}[Proof of Theorem \ref{thmA:roe alg of warped spaces}]
 By Lemma~\ref{lem:roe alg is ideal}, $\roestar{Y, \delta_\Gamma}$ is an ideal in $\fpstar{Y, \delta_\Gamma}$. Since $\roestar{Y, d}$ is contained in $\roestar{Y, \delta_\Gamma}$ and $\fpstar{\Gamma\curvearrowright Y}$ lies in $\fpstar{Y, \delta_\Gamma}$, it follows that $\fpstar{\Gamma\curvearrowright Y}\circ \roestar{Y, d}$ is contained in $\fpstar{Y, \delta_\Gamma}$, and so is its $\C$-span. By (the proof of) Proposition~\ref{prop:span generates fpstar}, we already know that
 $\roestar{Y, \delta_\Gamma}$ is contained in $\Gamma\cdot \roestar{Y, d}$ , so equality follows.

 For the other equality, it is enough to observe that $\Gamma\cdot \roecstar{Y,d}$ is contained in the norm closure of $\Gamma\cdot\roestar{Y, d}$; hence
 \[
  \overline{\Gamma\cdot\roecstar{Y, d}}^{\,\norm{\cdot}}
  =\overline{\Gamma\cdot\roestar{Y, d}}^{\,\norm{\cdot}}=\roecstar{Y, \delta_\Gamma}
 \]
by the first part of the proof.
\end{proof}

Considering the natural $*$-homomorphism
\begin{equation}\label{eq:varPsi}
    \begin{tikzcd}[row sep=0]
        \varPsi\colon &[-3em]  \fpstar{Y,d}\rtimesalg\Gamma \arrow[r] & B(L^2 Y), \\
                    &(T,\gamma) \arrow[|->,r] & T\pi_\alpha(\gamma),
    \end{tikzcd}
\end{equation}

where the crossed product structure is that defined by the conjugation in $ B(L^2 Y)$, we deduce the following corollary of Proposition~\ref{prop:span generates fpstar}.

\begin{cor}\label{cor:roe algebras are images of crossed products}
 Under Assumption~\ref{ass:Lipschitz action on proper}, if $(Y,d)$ has bounded geometry, then we have surjective $\ast$\=/homomorphisms
 \[
  \fpstar{Y,d}\rtimesalg\Gamma\to \fpstar{Y,\delta_\Gamma}
 \]
 \[
  \roestar{Y,d}\rtimesalg\Gamma\to \roestar{Y,\delta_\Gamma}.
 \]
 Taking the completions, we may thus identify $\fpcstar{Y,\delta_\Gamma}$ and $\roecstar{Y,\delta_\Gamma}$ with quotients of certain $C^*$-crossed products. Under Assumption~\ref{ass:ample module}, the latter algebra is the Roe algebra of $(Y,\delta_\Gamma)$.
\end{cor}

Under the above assumptions, the warped cone has bounded geometry, so we may consider its maximal Roe algebra. Observe that the action of $\Gamma$ on $\roestar{Y,d}$ extends to a $\Gamma$-action on $\roecstarmax{Y,d}$, so the full crossed product $\roecstarmax{Y,d}\rtimesmax\Gamma$ is well defined. The following is a formal consequence of the definitions.

\begin{cor}[Corollary~\ref{corA:warped roe as crossed product-maximal}]\label{cor:roe algebras are images of crossed products-maximal}
    With the same assumptions as in Corollary~\ref{cor:roe algebras are images of crossed products}, we have a surjective $\ast$-homomorphism
    \[
    \roecstarmax{Y,d}\rtimesmax\Gamma\to \roecstarmax{Y,\delta_\Gamma}.
    \]
\end{cor}
\begin{proof}
    The universal property of maximal Roe algebras implies that the $\ast$-homomorphism $\roestar{Y,d}\to \roestar{Y,\delta_\Gamma}$ induced by $\varPsi$ lifts to a $\ast$\=/homomorphism from $\roecstarmax{Y,d}$ to $\roecstarmax{Y,\delta_\Gamma}$. 
    Since $\varPsi$ is a homomorphism of $\roestar{Y,d}\rtimesalg\Gamma$ and $\roestar{Y,d}$ is dense in $\roecstarmax{Y,d}$, it is immediate that this extension defines a $\ast$-homomorphism 
    \[
    \roecstarmax{Y,d}\rtimesalg \Gamma\to\roecstarmax{Y,\delta_\Gamma},
    \]
    which then lifts to a $\ast$-homomorphism of the full crossed product by the universal property.
    Since $\roestar{Y,d}\rtimesalg\Gamma\to \roestar{Y,\delta_\Gamma}$ is surjective, the $\ast$-homomorphism thus defined has dense image, and is therefore surjective.
\end{proof}

\begin{rmk}
 Note that the above homomorphisms are basically never injective: It is usually the case that there are operators $T\in\fpstar{Y,d}$ and $\gamma\in\Gamma$ so that $T'\coloneqq T\pi_\alpha(\gamma)$ still belongs to $ \fpstar{Y,d}$. In this case, the difference $(T,\gamma)-(T',e)$ is in the kernel.
\end{rmk}

\begin{rmk}
    The Lipschitz condition in Assumption~\ref{ass:Lipschitz action on proper} could be weakened to $\alpha$ being an action by continuous coarse equivalences: The continuity condition is useful to prove Lemma~\ref{lem:the action normalizes}, while the coarse equivalence condition is enough to ensure that the warped metric still has bounded geometry (Remark~\ref{rmk:action by coarse equivalence suffices}). On the other hand, it is convenient to restrict to Lipschitz actions, as it implies both conditions at once. 
\end{rmk}

\section{ROE ALGEBRAS OF WARPED CONES} \label{sec:warped cones}
We now consider the relation between algebras of operators with finite dynamical propagation and Roe algebras of warped cones. 

First we need to recall some facts on tensor products, as covered in e.g.~\cite[Section 1.8]{willett2020higher}, \cite[Chapter 3]{brownozawa2008}. Given $C^*$-algebras $A$ and $B$, the algebraic tensor product $A\odot B$ is naturally equipped with a spatial norm defined as follows: Given two faithful representations $\pi_A\colon A\to  B(\CH_A)$ and $\pi_B\colon B\to  B(\CH_B)$, the formula
\[
 \pi_A\otimes\pi_B (a\otimes b) (\xi_A\otimes\xi_B)\coloneqq \pi_A(a)\xi_A\otimes \pi_B(b)\xi_B
\]
defines a tensor product representation $\pi_A\otimes\pi_B\colon A\odot B\to B(\CH_A\otimes\CH_B)$. The spatial norm is the pull-back of the operator norm on $B(\CH_A\otimes\CH_B)$. The completion of $A\odot B$ is the spatial tensor product, denoted $A\otimes B$.
Note that if $Y_A$ and $Y_B$ are locally compact topological spaces then
\(
C_0(Y_A\times Y_B)\cong C_0(Y_A)\otimes C_0(Y_B).
\)
It is also the case that $\CK(\CH_A\otimes\CH_B)\cong \CK(\CH_A)\otimes\CK(\CH_B)$,
where $\CK$ is the $C^*$-algebra of compact operators.

We now introduce (warped) cones. Let $(X,d)$ be a compact metric space of diameter at most $2$. For $t\geq 1$, let $d^t(x,y)\coloneqq td(x,y)$ be the rescaling of $d$ by $t$.
The \emph{open cone} is the metric space $\CO X\coloneqq (X\times \halfline,d_\CO)$, where
\[
 d_\CO\bigparen{(x_1,t_1),(x_2,t_2)}
 \coloneqq d^{t_1\wedge t_2}(x_1,x_2) + \abs{t_1-t_2},
\]
with $t_1\wedge t_2=\min\braces{t_1,t_2}$.
This formula defines a metric because $d$ has diameter at most $2$. The metric is proper because $X$ is compact.
Basic geometric information about (warped) cones can be found in \cite{SawickiThesis, Vig18}.

Let $\rho_X\colon C_0(X) = C(X)\to B(\CH_X)$ and $\rho_{\halfline}\colon C_0(\halfline)\to B(\CH_{\halfline})$ define geometric modules for $X$ and $\halfline$, respectively.
Topologically, the open cone $\CO X$ is homeomorphic to $X\times \halfline$; hence
\[
C_0(\CO X)=C_0(X\times \halfline)\cong C(X)\otimes C_0(\halfline).
\]
Consider $\CH_{\CO X} \coloneqq \CH_X \otimes \CH_{\halfline}$ with the tensor product representation $\rho_{\CO X}\coloneqq \rho_X\otimes\rho_{\halfline}\colon C_0(\CO X)\to\B\paren{\CH_{\CO X}}$.

\begin{lem}\label{lem:geometric OX module}
 The representation $\rho_{\CO X}\colon C_0(\CO X) \to  B(\CH_{\CO X})$ is a geometric module for $\CO X$. If both $\CH_X$ and $\CH_{\halfline}$ are faithful and at least one of them is ample, then $\CH_{\CO X}$ is ample.
\end{lem}
\begin{proof}
To show non-degeneracy, let $\xi \in \CH_{\CO X}$ such that $\rho_{\CO X}(f)\xi = 0$ for all $f \in C_0(\CO X)$.
Let $\xi'=\sum_{i=1}^n v_i \tensor w_i$ be a finite sum of elementary tensors with $\norm{ \xi - \xi' }<\epsilon$.
Then for every $g \in C(X)$ and $h \in C_0(\halfline)$, we have
\[
 \rho_{\CO X}(g\otimes h)\xi' =\sum_{i=1}^n \rho_{X}(g) v_i \tensor \rho_{\halfline}(h) w_i.
\]
and $\norm{\rho_{\CO X}(g\otimes h)\xi'}<\epsilon$.

Since the identity operators on $ B(\CH_X)$ and on $B(\CH_{\halfline})$ belong to the strong closures of $C(X)$ and $C_0(\halfline)$, we can pick $g$ and $h$ such that $\norm{\rho_{X}(g) v_i - v_i}<\epsilon'$ and $\norm{\rho_{\halfline}(h )w_i - w_i}<\epsilon'$ for all $i$.
Choosing $\epsilon'$ sufficiently small, it follows that $\norm{\sum_{i=1}^n \rho_{X}(g) v_i \tensor \rho_{\halfline}(h) w_i - \xi'} < \epsilon$. Therefore, $\norm {\xi'} < 2\epsilon$. It follows that $\norm{\xi}<3\epsilon$. Since this holds for all $\epsilon > 0$, we conclude that $\xi = 0$, and $\CH_{\CO X}$ is indeed non-degenerate.

Let now $\CH_X$ and $\CH_{\halfline}$ be faithful, and assume that one of them is ample. Fix $f \in C_0(\CO X)$ non-zero.
Then $f\bar f$ is non-negative on $\CO X$, and there are some non-empty open sets $U\subseteq X$ and $V\subseteq \halfline$ with $f\bar f \geq c$ on $U\times V$ for some $c>0$.
Find non-zero $g\in C(X)$ and $h\in C_0(\halfline)$ supported on $U$ and $V$ respectively.
Then there is $f' \in C_0(\CO X)$ with $f\bar f f' = g \tensor h$.
Now, if $\rho_{\CO X}(f)$ was compact, then also also $\rho_{\CO X} (g\tensor h)=\rho_{\CO X}(f)\rho_{\CO X} (\bar ff')$ would be compact.
But $\rho_{\CO X}(g \tensor h)$ is the tensor product of a non-zero and a non-compact operator, and therefore not compact.
\end{proof}

Note that if $\mu$ is a measure of full support on $X$ (i.e.~so that non-empty open sets have positive measure), then $L^2 X$ is a faithful representation of $C(X)$. Since $L^2\halfline$ (with the Lebesgue measure $\lambda$) is an ample $\halfline$-module, Lemma~\ref{lem:geometric OX module} has the following consequence.

\begin{cor}
 Suppose that $\mu$ has full support on $X$. Then the representation $M_\bullet\colon C_0(\CO X)\to L^2(X\times \halfline, \mu\times \lambda)$ given by multiplication operators is an ample geometric $\CO X$-module.
\end{cor}

Now, let $\alpha\colon \Gamma\curvearrowright X$ be a continuous action of a countable group $\Gamma$ with a symmetric proper length function $\ell$. The action $\alpha\times \id$ extends $\alpha$ to a continuous action on $\CO X$.

\begin{de}
The \emph{unified warped cone} is the metric space $(\CO_\Gamma X,\delta_\Gamma)$, obtained by warping the metric $d_{\CO X}$ along the $\Gamma$-action $\alpha\times \id$.
\end{de}

As usual, warping the metric does not change the underlying topology. Hence, the geometric $\CO X$-module $\CH_{\CO X}$ is a geometric $\CO_\Gamma X$-module as well, and ampleness is also preserved.

We now study Roe algebras of warped cones. From now on, we will work under the following assumption.

\begin{assumption}\label{ass:warped cones}
 $(X,d,\mu)$ is a compact metric space of diameter at most $2$ with a $\sigma$-finite measure of full support, and $\alpha\colon\Gamma\curvearrowright X$ is a non-singular Lipschitz action.
\end{assumption}

In this setting, $\CH_{\CO X}\coloneqq L^2(X\times \halfline)$ is an ample $\CO_\Gamma X$-module with the representation $M_\bullet$ by multiplication operators, and both Assumption~\ref{ass:Lipschitz action on proper} and \ref{ass:ample module} are satisfied. This can now be used to construct the Roe algebra $\roecstar{\CO_\Gamma X}$, and all the results proved in the previous sections hold true. We use notation analogous to Section~\ref{sec:roe alg warped spaces} as follows.

\begin{notation}\label{notation module cones}
 We set the geometric modules $\CH_{\CO X}$ and $\CH_{\CO_\Gamma X}$ to be $L^2(X\times \halfline)$ and simply write $\fpstar{\CO X}$, $\fpcstar{\CO X}$, $\roestar{\CO X}$ and $\roecstar{\CO X}$ for the associated subalgebras of $ B(L^2(X\times \halfline))$. We use analogous conventions for $\CO_\Gamma X$ and we also let $\CK(\CO X)=\CK(\CO_\Gamma X)=\CK(L^2(X\times \halfline))$ be the compact operators.
\end{notation}

Since the warped cone is defined by warping the cone $\CO X$ along the action $\alpha\times\id\colon\Gamma\curvearrowright\CO X$, Proposition~\ref{prop:span generates fpstar} and Theorem~\ref{thmA:roe alg of warped spaces} imply that
\begin{align*}
    \fpstar{\CO_\Gamma X} &= \pi_{\alpha\times\id}(\Gamma)\cdot \fpstar{\CO X},\\
    \fpcstar{\CO_\Gamma X} &= \overline{\pi_{\alpha\times\id}(\Gamma)\cdot\fpcstar{\CO X}}^{\,\norm{\cdot}},\\
    \roestar{\CO_\Gamma X} &= \pi_{\alpha\times\id}(\Gamma)\cdot \roestar{\CO X},\\
    \roecstar{\CO_\Gamma X} &= \overline{\pi_{\alpha\times\id}(\Gamma)\cdot\roecstar{\CO X}}^{\,\norm{\cdot}}.
\end{align*}

It is interesting to revisit these identities in terms of tensor products. Namely, consider the following tensor product homomorphism of the spatial tensor product:
\[
  B(L^2 X)\otimes B(L^2\halfline)\to  B(L^2 X\otimes L^2\halfline)= B(\CH_{\CO X}).
\]
This restricts to a $\ast$\=/homomorphism of the algebraic tensor product
 \[
 \fpstar{\Gamma\curvearrowright X}\odot \fpstar{L^2\halfline}\to \fpstar{\CO_\Gamma X}.
 \]
Here we can also recognise the Koopman representation of the action $\alpha\colon\Gamma\curvearrowright X$, noting that $\pi_{\alpha\times \id}(\gamma)=\pi_{\alpha}(\gamma)\otimes 1$.

On the other hand, the tensor product interpretation does not interact well with the assumptions of local compactness. Namely, in order to obtain a homomorphism into $\roestar{\CO_\Gamma X}$, one is forced to restrict to
 \[
 \bigparen{\fpstar{\Gamma\curvearrowright X}\cap\CK(L^2 X)}\odot \roestar{L^2\halfline}\to \roestar{\CO_\Gamma X}.
 \]
However, the image of this homomorphism ``sees'' very little of the Roe $*$\=/algebra $\roestar{\CO_\Gamma X}$ in general. For instance, if the action is ergodic but not strongly ergodic it follows from Corollary~\ref{corA:strong ergodicity} that $\fpstar{\Gamma\curvearrowright X}\cap\CK(L^2 X)=\{0\}$, so the homomorphism is trivial.

We now aim to understand in the context of warped cones the kernel of the $*$-homomorphism
\[
\varPsi\colon \fpstar{Y,d}\rtimesalg\Gamma \to B(L^2 Y),
\]
as defined in \eqref{eq:varPsi} (see Corollary~\ref{cor:roe algebras are images of crossed products}). 
To do this, we assume an additional regularity condition on the open cone, namely the operator norm localization property from \cite{CTWY}. 

\begin{de}
    A $Y$-module $\CH_Y$ has the \emph{operator norm localization property} (\emph{ONL property}) with constant $0<c\leq 1$ if for every $r>0$, there is $R=R(r)$ such that if $T\in B(\CH_Y)$ is an operator of propagation at most $r$, then there exists a measurable $A\subseteq Y$ with $\diam(A)\leq R$ and $\norm{\chi_AT}\geq c\norm{T}$.
\end{de}

\begin{rmk}
    The above is a slight rephrasing of {\cite[Definition 2.2]{CTWY}} applied to the adjoint of $T$. Note that ample modules of most ``reasonable'' spaces have the ONL property. In particular, this is the case for $\CO X$ whenever $X$ is a compact Riemannian manifold. In fact, it is easy to show that such a cone has property A, and hence the ONL property. This can be shown directly, by using that the ball of radius $r$ around a point $(x,t)\in \CO X$ looks more and more like a ball in $\RR^{\dim(X)}\times \RR$ as $t$ increases (they are bi-Lipschitz equivalent with arbitrarily good constants). Alternatively, one can also observe that $\CO X$ has asymptotic dimension $\dim(X)+1$ and use that finite asymptotic dimension implies property A; see e.g.\ \cite[Section 11.5]{Roe03}.
    We refer to \cite{CTWY,BNSW,Sak} for more information on the ONL property and its relation with e.g.~Yu's property A.
\end{rmk}

Given a cone $\CO X$ and $1\leq a \leq b\leq \infty$, let $\chi_{[a,b]}\in B(\CH_{\CO X})$ be the projection associated with the truncated cone $X\times [a,b]\subseteq \CO X$.
We need a preliminary observation.

\begin{lem}\label{lem:locally cpt not cpt}
    Let $T\in B(\CH_{\CO X})$ be locally compact. If $T$  is not compact, there is an $\epsilon>0$ such that $\norm{T\chi_{[t,\infty]}} \geq \epsilon$ for every $t\geq 1$.
\end{lem}
\begin{proof}
    Note that $T= T\chi_{[1,t)}+T\chi_{[t,\infty)}$. The first summand is always a compact operator by the local compactness assumption. If the norm of the second summand goes to $0$ as $t$ increases, this would show that $T$ is norm-approximated by compact operators, and therefore compact.
\end{proof}

Now we can prove Theorem \ref{thmA:roe warped cone}, which we restate for convenience.

\begin{thm}[Theorem \ref{thmA:roe warped cone}]\label{thm:roe alg of warped cone}
 Under Assumption~\ref{ass:warped cones}, if the $\Gamma$-action is free, $\CO X$ has bounded geometry and $\CH_{\CO X}$ has the ONL property, then 
 \[
 \varPsi\colon\roestar{\CO X}\rtimesalg \Gamma\to\roestar{\CO_\Gamma X}\]
 descends to an isomorphism
 \begin{equation}\label{eq: thm6.8 part 1}
  \overline\varPsi\colon
  \frac{\roestar{\CO X}}{\CK\paren{L^2(\CO X)}\cap \roestar{\CO X}}
  \rtimesalg \Gamma
  \to \frac{\roestar{\CO_\Gamma X}}{\CK\paren{L^2(\CO X)}\cap \roestar{\CO_\Gamma X}}     
 \end{equation}
 and to an injective $\ast$-homomorphism with dense image:
 \begin{equation}\label{eq: thm6.8 part 2}
  \frac{\roecstar{\CO X}}{\CK\paren{L^2(\CO X)}}
  \rtimesalg \Gamma
  \to \frac{\roecstar{\CO_\Gamma X}}{\CK\paren{L^2(\CO X)}}.
 \end{equation}
 The quotient $\frac{\roecstar{\CO_\Gamma X}}{\CK\paren{L^2(\CO X)}}$ may thus be viewed as a completion of $\frac{\roecstar{\CO X}}{\CK\paren{L^2(\CO X)}}\rtimesalg \Gamma$.
\end{thm}
\begin{proof}
 We start by examining the Roe $*$-algebra $\roestar{\CO_\Gamma X}$.
 Corollary~\ref{cor:roe algebras are images of crossed products} implies that the map
 \[
 \roestar{\CO X}\rtimesalg \Gamma
  \to \frac{\roestar{\CO_\Gamma X}}{\CK\paren{L^2(\CO X)}\cap \roestar{\CO_\Gamma X}}
 \]
 is surjective. Since the compact operators are a two-sided ideal in $B(\CH_{\CO X})$, the algebraic crossed product $\frac{\roestar{\CO X}}{\CK\paren{L^2(\CO X)}\cap \roestar{\CO X}}\rtimesalg \Gamma$ is well-defined.
 Given $T\in\CK\paren{L^2(\CO X)}\cap \roestar{\CO X}$ and $\gamma\in\Gamma$, the image $\varPsi(T,\gamma)=T\pi(\gamma)$ is still a compact operator and it has finite propagation with respect to the warped metric. That is, $\varPsi(T,\gamma)\in\CK\paren{L^2(\CO X)}\cap \roestar{\CO_\Gamma X}$. This shows that $\varPsi$ does indeed descend to a surjective $\ast$-homomorphism. 
 
 It remains to prove that whenever $\varPsi$ maps a finite sum $\sum_{\gamma\in S}(T_\gamma,\gamma)$ into $\CK\paren{L^2(\CO X)}\cap \roestar{\CO_\Gamma X}$, then $T_\gamma$ must be compact for every $\gamma\in S$.
 Fix a finite sum $\sum_{\gamma\in S}(T_\gamma,\gamma)$ and let
 \[
    W\coloneqq\Psi\Bigparen{\sum_{\gamma\in S}(T_\gamma,\gamma)}.
 \]
 We show that if there is $\bar\gamma\in S$ such that the operator $T_{\bar \gamma}$ is not compact, then $W$ is also not compact. 
 
 Suppose that each $T_\gamma$ has propagation at most $r$ and let $R=R(r)$ be the constant given by the ONL property of $\CO X$.
 By Lemma~\ref{lem:locally cpt not cpt}, there is $\epsilon>0$ such that $\norm{T_{\bar \gamma}\chi_{[t,\infty)}}\geq \epsilon$ for every $t\geq 1$. We may then find $A_0\subseteq\CO X$ with $\diam_{\CO X}(A_0)\leq R$ and $\norm{\chi_{A_0}T_{\bar\gamma}}\geq c\norm{T_{\bar\gamma}}\geq c\epsilon$.

 Let $p_X$ and $p_\RR$ denote the coordinate projections from $\CO X$ to $X$ and $\halfline$, and let $t_1\coloneqq \sup(p_\RR(A_0))+1$. We may then apply the ONL property to the operator $T\chi_{[t_1,\infty)}$ to obtain $A_1\subset X \times [t_1,\infty)$ of diameter at most $R$ and $\norm{\chi_{A_1}T_{\bar\gamma}}\geq c\epsilon$.
 Repeating this procedure, we obtain a sequence of measurable subsets $A_n\subset \CO X$ satisfying:
 \begin{enumerate}[(i)]
     \item $\diam_{\CO X}(A_n)\leq R$;
     \item $p_\RR(A_n)\to\infty$;
     \item $A_n\cap A_m=\emptyset$ for every $n\neq m$;
     \item $\norm{\chi_{A_n}T_{\bar\gamma}}\geq c\epsilon$.
 \end{enumerate}

 Let $C_n\coloneqq \overline{N_{\CO X}(A_n; r)}$ be the closed $r$-neighbourhood in the open cone. Since  for each $\gamma \in S$ the operator $T_{\gamma}$ has propagation at most $r$, it follows that $\chi_{A_n}T_\gamma=\chi_{A_n}T_\gamma\chi_{C_n}$. 
 We obtain
 \begin{equation}\label{eq:final_thm_1}
    \chi_{A_n}W
    =\sum_{\gamma\in S}\chi_{A_n}T_\gamma\pi(\gamma)
    =\sum_{\gamma\in S}\chi_{A_n}T_\gamma\chi_{C_n}\pi(\gamma)
    =\sum_{\gamma\in S}\chi_{A_n}T_\gamma\pi(\gamma)\chi_{\gamma^{-1}\cdot C_n}.
 \end{equation}
 Now, condition (ii) implies that $\diam_X(p_X(C_n))\to 0$ as $n$ tends to infinity.
 Since the action is free and $X$ is compact, for every $n$ large enough the set $\bar\gamma\cdot p_X(C_n)$ is disjoint from $\gamma\cdot (p_X(C_n))$ for every $\gamma\in S\smallsetminus \{\bar\gamma\}$. A fortiori, $\bar\gamma\cdot C_n$ is then disjoint from $\gamma\cdot C_n$, so \eqref{eq:final_thm_1} implies that
 \begin{equation*}
     \chi_{A_n}W\chi_{\bar\gamma^{-1}\cdot C_n}
     = \chi_{A_n}T_{\bar \gamma}\pi(\gamma)\chi_{\bar\gamma^{-1}\cdot C_n}
     = \chi_{A_n}T_{\bar \gamma}\pi(\gamma).
 \end{equation*}

 Since $\pi(\bar\gamma)$ is a unitary, it then follows that
 \[
 \norm{\chi_{A_n}W}
 \geq
  \norm{\chi_{A_n}W\chi_{\bar\gamma^{-1}\cdot C_n}}
  \geq
  \norm{\chi_{A_n}T_{\bar \gamma}\pi(\bar \gamma)}
  =
  \norm{\chi_{A_n}T_{\bar \gamma}}
  \geq c\epsilon
 \]
 for every $n$ large enough. By (iii), the projections $\chi_{A_n}$ are orthogonal to one another, so this implies that $W$ is not compact.

 \

 Turning to the Roe $C^*$-algebras, we argue as above to see that 
 \[
 \roecstar{\CO X}\rtimesalg \Gamma
  \to \frac{\roecstar{\CO_\Gamma X}}{\CK\paren{L^2(\CO X)}}
 \]
 has dense image by Corollary~\ref{cor:roe algebras are images of crossed products},
 the algebraic crossed product $\frac{\roecstar{\CO X}}{\CK\paren{L^2(\CO X)}}\rtimesalg \Gamma$ is well-defined (observe that the compact operators always belong to the Roe algebra), and $\varPsi$ descends to a $\ast$-homomorphism.

 The verification that the induced homomorphism is injective follows along the same lines, with one additional small step:
 Once the finite sum $\sum_{\gamma\in S}(T_\gamma,\gamma)$ with $T_{\bar \gamma}$ non-compact is fixed, we need to choose approximating operators $T_\gamma'$ that have finite propagation and $\norm{T_\gamma-T_\gamma'}$ is much smaller than $c\epsilon$. It is then the case that $\norm{T'_\gamma\chi_{[t,\infty)}}\geq \epsilon/2$ for every $t\geq 1$.
 We may then apply the same argument of above to deduce that $\chi_{A_n}\varPsi\Bigparen{\sum_{\gamma\in S}(T_\gamma',\gamma)}$ has norm at least $c\epsilon/2$ for large enough $n$, and deduce that $\chi_{A_n}W$ has norm at least $c\epsilon/4$.
\end{proof}

\begin{cor}[{Corollary~\ref{corA:roe alg of warped cones-maximal}}]\label{cor:roe alg of warped cones-maximal}
Under Assumption~\ref{ass:warped cones}, if the $\Gamma$-action is free, $\CO X$ has bounded geometry and $\CH_{\CO X}$ has the ONL property, $\varPsi$ induces a $\ast$-isomorphism
    \[
      \frac{\roecstarmax{\CO X}}{\CK\paren{L^2(\CO X)}}
      \rtimesmax \Gamma
      \to \frac{\roecstarmax{\CO_\Gamma X}}{\CK\paren{L^2(\CO X)}}.
    \]
\end{cor}
\begin{proof}
    Consider the $\ast$-isomorphism
    \[
      \overline\varPsi\colon
      \frac{\roestar{\CO X}}{\CK\paren{L^2(\CO X)}\cap \roestar{\CO X}}
      \rtimesalg \Gamma
      \xrightarrow{\cong} \frac{\roestar{\CO_\Gamma X}}{\CK\paren{L^2(\CO X)}\cap \roestar{\CO_\Gamma X}}.
     \]
    We may argue as above or as in the proof of Corollary~\ref{cor:roe algebras are images of crossed products-maximal} to deduce that it induces a $*$\=/homomorphism
    \[
      \frac{\roecstarmax{\CO X}}{\CK\paren{L^2(\CO X)}}
      \rtimesmax \Gamma
      \to \frac{\roecstarmax{\CO_\Gamma X}}{\CK\paren{L^2(\CO X)}}.
    \]
    
    We will now use the universal property to construct an inverse for it.
    The $\ast$-homomorphism of $\roestar{\CO_\Gamma X}$ obtained composing $\overline\varPsi^{-1}$ with the embedding into the full crossed product extends to a $\ast$-homomorphism
    \[
      \roecstarmax{\CO_\Gamma X}\to
      \frac{\roecstarmax{\CO X}}{\CK\paren{L^2(\CO X)}}
      \rtimesmax \Gamma.
    \]
    This homomorphism sends compact operators to compact operators, so it descends to the quotient $\roecstarmax{\CO_\Gamma X}/\CK(\CO X)$. By the uniqueness part of the universal property for maximal norms, it follows that the map thus constructed is the required inverse.
\end{proof}

\begin{rmk}
    Equation \eqref{eq: thm6.8 part 1} in Theorem~\ref{thm:roe alg of warped cone} can be proved without the assumption that $\CH_{\CO X}$ has the ONL property. Namely, in the second part of the proof one can use the bounded geometry assumption on $\CO X$ to partition $\CO X$ and directly show that if $T_{\bar\gamma}$ is not compact, then there are sets $A_n$ satisfying (i)--(iii) and a version of (iv) in which the uniform constant $c$ is replaced by some other constant $\delta>0$ possibly depending on $T_{\bar\gamma}$. This suffices to complete the proof of \eqref{eq: thm6.8 part 1}. On the other hand, the argument we outlined for  \eqref{eq: thm6.8 part 2} does make use of the uniform constant $c$ when choosing the finite propagation operator $T'_{\bar \gamma}$ approximating $T_{\bar \gamma}$, so it does rely on the ONL property.
\end{rmk}

\end{document}